\documentclass[11pt,letterpaper]{article}
\usepackage{amsfonts, amsmath, amssymb, amscd, amsthm, color, graphicx, mathabx, mathrsfs, wasysym, setspace, mdwlist, calc, float, mathtools, tikz-cd}
\usepackage{setspace}

\usepackage{enumitem}
\usepackage{tocloft}
\usepackage{xcolor}

\usepackage{hyperref}
\hypersetup{linktocpage}

\hypersetup{colorlinks,
    linkcolor={red!50!black},
    citecolor={blue!80!black},
    urlcolor={blue!80!black}}
\usepackage{float}

\usepackage{xcolor}

\usepackage{hyperref}

\hypersetup{colorlinks,
    linkcolor={red!50!black},
    citecolor={blue!80!black},
    urlcolor={blue!80!black}}
\usepackage{float}

\newcommand\blfootnote[1]{%
  \begingroup
  \renewcommand\thefootnote{}\footnote{#1}%
  \addtocounter{footnote}{-1}%
  \endgroup
}

\renewcommand{\phi}{\varphi}

\renewcommand{\d }{{\rm d} }
\renewcommand{\dh}{\widehat \d}
\newcommand{\dl }{\dh_{H_i} }
\newcommand{\h}{\hookrightarrow_h}
\newcommand{\G }{\Gamma (G, \mathcal A)}

\newcommand{\Hl }{\{ H_i \} _{i\in I } }

\newcommand{\e }{\varepsilon }

\renewcommand{\kappa }{\varkappa}

\newcommand{\lab}{{\bf Lab}}
\renewcommand{\ll }{\langle\hspace{-.7mm}\langle }
\newcommand{\rr }{\rangle\hspace{-.7mm}\rangle }

\newcommand{\NN}{\mathbb N}
\newcommand{\ZZ}{\mathbb Z}

\newcommand{\di}{\widehat{\d}_{H_i}}
\newcommand{\da }{\d_{\mathcal A}\, }
\newcommand{\vk}{\varkappa}
\newcommand{\Ker}{\operatorname{Ker}}

\newtheorem{thm}{Theorem}[section]
\newtheorem*{thm*}{Theorem}
\newtheorem{cor}[thm]{Corollary}
\newtheorem{lem}[thm]{Lemma}

\newtheorem{prop}[thm]{Proposition}
\newtheorem{prob}[thm]{Question}

\theoremstyle{definition}
\newtheorem{defn}[thm]{Definition}
\newtheorem{conv}[thm]{Convention}
\newtheorem{ex}[thm]{Example}

\theoremstyle{remark}
\newtheorem{rem}[thm]{Remark}

\begin{document}

\title{Small cancellation and outer automorphisms of Kazhdan groups acting on hyperbolic spaces}

\author{I. Chifan, A. Ioana, D. Osin, B. Sun}

\date{}

\maketitle

\begin{abstract}
We show that every finite group realizes as the outer automorphism group of a hyperbolic group with Kazhdan property (T) and trivial finite radical. This result complements the well-known theorem of Paulin stating that the outer automorphism group of every hyperbolic group with property (T) is finite. We also show that, for every countable group $Q$, there exists an acylindrically hyperbolic group $G$ with property (T) such that $Out(G)\cong Q$. The proofs employ strengthened versions of some previously known results in small cancellation theory.
\end{abstract}

\tableofcontents

\section{Introduction}
It has long been known that the combination of hyperbolic geometry and property (T) implies rigidity of the algebraic structure of the group. The ultimate manifestation of this phenomenon is the following theorem.

\begin{thm}[Paulin, \cite{Pau}]
The outer automorphism group of a hyperbolic group with property (T) is finite.
\end{thm}

Informally speaking, Paulin shows that every group $G$ admitting infinitely many ``sufficiently distinct" isometric actions on a hyperbolic space also acts on an $\mathbb R$-tree without globally fixed points. For a hyperbolic group $G$, these ``sufficiently distinct" actions can be constructed by precomposing the natural action of $G$ on its Cayley graph with elements of $Out(G)$. The relevance of property (T) in this context stems from the observation that every action of a Kazhdan group on an $\mathbb R$-tree has a fixed point (see \cite{HV}).

Recall that every hyperbolic group $G$ contains a unique maximal finite normal subgroup denoted by $K(G)$ and called the \emph{finite radical} of $G$ (see, for example, \cite{Ols93}). It is natural to ask the following question reminiscent of the inverse Galois problem.

\begin{prob}\label{Galois}
Which finite groups can be realized as $Out(G)$ for a hyperbolic group $G$ with property (T) and trivial finite radical?
\end{prob}

Without the condition $K(G)=\{ 1\}$, Question \ref{Galois} admits an easy answer since every finite group can be realised as the group of outer automorphisms of another finite group \cite{Cor}. One can also consider direct products of the form $G=K\times H$, where $K$ is an appropriate finite group and $H$ is a non-trivial, torsion-free, hyperbolic group with property (T) such that $Out(H)=\{ 1\}$. Obviously, we have $Out(G)\cong Out (K)$ in this case. However, these examples are not satisfactory since the outer automorphisms of such a group $G$ do not reflect the symmetries of its ``truly hyperbolic" part.

We give a complete answer to Question \ref{Galois} by proving the following.

\begin{thm}\label{app1}
For every finite group $Q$, there exists a hyperbolic group $G$ with property (T) and trivial finite radical such that $Out(G)\cong Q$.
\end{thm}

Another natural question prompted by Paulin's theorem and its generalizations to relatively hyperbolic groups \cite{BS,DS} is whether the existence of a ``nice" action of a group $G$ on a hyperbolic space imposes any restrictions on $Out(G)$. We address this question for the class of acylindrically hyperbolic groups, which was introduced in \cite{Osi16} and received considerable attention in recent years (see \cite{Osi18} and references therein). Informally, a group $G$ is \emph{acylindrically hyperbolic} if it admits a non-elementary action on a hyperbolic space $S$ such that the induced action of $G$ on $S\times S$ satisfies a certain properness assumption. For the precise definition, we refer the reader to Section \ref{Sec:AH}.

\begin{thm}\label{app2}
For every countable group $Q$, there exists a finitely generated acylindrically hyperbolic group $G$ with property (T) such that $Out(G)\cong Q$.
\end{thm}

\begin{rem}
Every acylindrically hyperbolic group also contains a unique maximal finite normal subgroup \cite[Theorem 2.24]{DGO}. The same terminology and notation as in the hyperbolic case is used here. Our approach allows us to additionally ensure that the group $G$ from Theorem \ref{app2} has trivial finite radical (see Theorem \ref{app3}). Note, however, that Theorem \ref{app2} (unlike Theorem \ref{app1}) is new even without the condition $K(G)=\{ 1\}$.
\end{rem}

Theorem \ref{app2} strengthens several previously known results, all of which were obtained using some sort of small cancellation theory. Minasyan \cite{Min} prowed that every countable group realizes as $Out(G)$ for a property (T) group $G$.  However, groups constructed in \cite{Min} are far from being acylindrically hyperbolic. Ollivier and Wise \cite{OW} showed that every countable group $Q$ embeds in $Out(G)$ for some acylindrically hyperbolic group $G$ with property (T). (Although acylindrical hyperbolicity is not mentioned in \cite{OW}, groups constructed in \cite{OW} satisfy a small cancellation condition, which implies acylindrical hyperbolicity by the work of Gruber and Sisto \cite{GS}.) In addition, one can ensure that $Q$ is of finite index in $Out(G)$ if $Q$ is finitely generated using the approach suggested in \cite{BO}. However, the technique used in \cite{BO} is not sufficient to prove the equality $Out(G)=Q$. A significant part of this paper is devoted to the development an improved version of small cancellation theory that allows us to obtain a more precise result. Since our work in this direction is likely to have other applications, we discuss it below in more detail.

The classical small cancellation theory studies quotient groups of the form $$G=F(X) /\ll \mathcal R\rr,$$ where $F(X)$ is the free group with basis $X$ and $\mathcal R$ is a set of reduced words in the alphabet $X\cup X^{-1}$ with ``small overlaps". The central result about such groups is the Greendlinger lemma stating that every reduced word in $X\cup X^{-1}$ representing the identity in $G$ contains a ``long" subword of some relation $R\in \mathcal R$.

A generalization of this theory to quotients of groups acting on hyperbolic spaces was suggested by Gromov \cite{Gro} and elaborated by Olshanskii \cite{Ols93}. The main technical tools and ideas employed in \cite{Ols93} go back to the geometric method of studying groups via van Kampen diagrams developed by Olshanskii in the late 1970s and early 1980s, which enabled him to construct examples of groups with unexpected properties and to give a short proof of the Novikov-Adyan theorem \cite{Adi} on groups of finite exponent. For more details, we refer to \cite{Ols-book} and references therein.

Although the paper \cite{Ols93} only deals with hyperbolic groups, many results obtained there hold for relatively hyperbolic and, more generally, acylindrically hyperbolic groups with little modifications, see \cite{Hull,Osi10}. In these settings, the Greendlinger lemma is no longer true and is replaced with a weaker conclusion, which is formalized using the key notion of a \emph{contiguity subdiagram} suggested by Olshanskii \cite{Ols-book, Ols93}. In Section \ref{Sec:WSCC}, we propose a new small cancellation condition, which allows us to eliminate contiguity subdiagrams and obtain a stronger result similar to the original Greendlinger lemma (see Proposition \ref{Prop:GL}). At the core of our approach are hyperbolically embedded collections of subgroups and a new notion of an \emph{attracting geodesic}, which seems to be of independent interest. We also take the opportunity to generalize and strengthen some results of \cite{Hull,Ols93, Osi10}.

The paper is organized as follows. In the next section, we review the necessary background on hyperbolic groups and their generalizations. In Section 3, we develop the small cancellation toolbox necessary for proving Theorems \ref{app1}, and \ref{app2}. The proofs of the latter two theorems are given in Section 4.

\paragraph{Acknowledgments.} I. Chifan was supported by the NSF grants DMS-1854194 and DMS-2154637.  A. Ioana was supported by the NSF grants DMS-1854074 and DMS-2153805, and a Simons Fellowship. D. Osin was supported by the NSF grant DMS-1853989. We would also like to thank the anonymous referee for the careful reading of our manuscript and useful remarks.

\section{Preliminaries}\label{Sec:GTPrelim}

\subsection{Hyperbolic spaces}
Recall that metric space $S$ with a distance function $\d$ is said to be \emph{geodesic}, if every two points $a,b\in S$ can be connected by a path of length $\d(a,b)$. A geodesic metric space $S$ is  \emph{$\delta$-hyperbolic} for some $\delta\ge 0$ if, for any geodesic triangle $\Delta $ in $S$, every side of $\Delta $ is contained in the union of the closed $\delta$-neighborhoods of the other two sides \cite{Gro}.

For a path $p$ in a metric space, we denote by $p_-$ and $p_+$ its origin and terminal point, respectively. If $p$ is rectifiable, we denote by $\ell(p)$ its length. We will need the following standard results about geodesic polygons in hyperbolic spaces.

\begin{lem}\label{Lem:HRect}
Let $(S,\d)$ be a $\delta $-hyperbolic space, $Q=pqrs$ a quadrilateral in $S$ with geodesic sides $p$, $q$, $r$, $s$.
\begin{enumerate}
\item[(a)] Every side of $Q$ belongs to the union of the closed $2\delta$-neighborhoods of the other three sides.
\item[(b)] For any point $x\in p$, we have $\d(x, r)\le 2\delta +\max\{ \ell(q), \ell(s)\}$.
\item[(c)] Suppose that $\ell(p)> \ell (q)+\ell(s)+4\delta $. Then there exists a subpath $t$ of $p$ such that $$\ell(t)=\ell(p)-\ell (q)-\ell(s)-4\delta$$ and $\max\{ \d(t_{-}, r), \d(t_{+}, r)\}  \le 2\delta$.
\end{enumerate}
\end{lem}

\begin{proof}
The first part can be easily derived from the definition of a hyperbolic space by drawing a diagonal in $Q$. Part (b) follows from (a) by the triangle inequality. To prove (c), it suffices to take the subpath $t$ of $p$ such that $\d(p_-, t_-)= 2\delta + \ell(q)$ and $\d(p_+, t_+)= 2\delta + \ell(s)$ (see Fig. \ref{fig3b1}). Since $\ell(p)> \ell (q)+\ell(s)+4\delta $, the distances from $t_-$ and $t_+$ to the union of the interiors of $s$ and $q$ is greater than $2\delta$. Combining this with part (a), we obtain $\max\{ \d (t_-, r), \d(t_+, r)\} \le 2\delta$.
\end{proof}

The next lemma can be thought of as a generalization of the previous one to more general polygons. It is a simplification of \cite[Lemma 10]{Ols93}.

\begin{lem}[Olshanskii]\label{N123}
Let $(S,\d)$ be a $\delta $-hyperbolic space. Suppose that the set of all sides of a geodesic polygon $P=p_1 p_2\ldots p_n$ is partitioned into two subsets $A$ and $B$. Let $\alpha$ (respectively $\beta$) denote the sum of lengths of sides from $A$ (respectively $B$). Assume, in addition, that $\alpha > \max\{c n, 10^3\beta \}$ for some $c \ge 3\cdot  10^4\delta$. Then there
exist two distinct sides $p_i, p_j\in A$ and a subpath $t$ of $p_i$ of length greater than $10^{-3}c$ such that $$\max\{ \d(t_-, p_j), \d(t_+, p_j)\} \le 13\delta.$$
\end{lem}

Throughout this paper, we often think of graphs as metric spaces. Given a connected graph $\Gamma$, we identify every open edge of $\Gamma$ with $(0,1)$ and define the distance between two points $a,b\in \Gamma$ to be the length of the shortest path in $\Gamma$ connecting $a$ to $b$.

\begin{figure}
   \begin{center}
   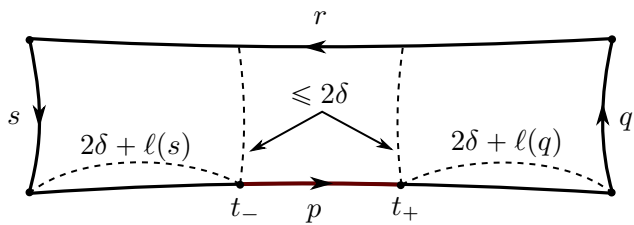
   \end{center}
   \vspace{-3mm}
   \caption{Proof of Lemma \ref{Lem:HRect} (c).}\label{fig3b1}
 \end{figure}

\begin{rem}
If $\Gamma $ is a hyperbolic graph, we can (and will) always assume that its hyperbolicity constant is a non-negative integer. Further, if the sides of the geodesic polygons considered in Lemma \ref{Lem:HRect} and Lemma \ref{N123} are (combinatorial) paths starting and ending at some vertices of $\Gamma$, elementary arguments show that we can additionally require the endpoints of the path $t$ in each of these lemmas to be vertices of $\Gamma$ as well.
\end{rem}

A group $G$ is {\it hyperbolic} if it is generated by a finite set $X$ and its Cayley graph $\Gamma (G,X)$ is a hyperbolic metric space. This definition is independent of the choice of a particular finite generating set $X$. A hyperbolic group is called \emph{elementary} if it contains a cyclic subgroup of finite index. For examples and basic properties of hyperbolic groups, we refer the reader to Chapters III.H and III.$\Gamma$ of \cite{BH}

\subsection{Relative hyperbolicity and hyperbolically embedded subgroups} \label{Sec:Hyp}

Hyperbolically embedded collections of subgroups play a crucial role in our paper. To formulate the definition, it is convenient to work with generating alphabets instead of generating sets of groups. By a \emph{generating alphabet} $\mathcal A$ of a group $G$ we mean an abstract set given together with a map $\mathcal A\to G$ whose image generates $G$; to simplify our notation, we do not distinguish between elements of $\mathcal A$ and their images in $G$ whenever no confusion is possible.

By the \emph{Cayley graph} of $G$ with respect to a generating alphabet $\mathcal A$, denoted $\Gamma (G, \mathcal A)$, we mean a graph with the vertex set $G$ and the set of edges defined as follows. For every $a\in \mathcal A$ and every $g\in G$, there is an oriented edge $e$ going from $g$ to $ga$ in $\Gamma (G, \mathcal A)$ and labelled by $a$. Given a combinatorial path $p$ in $\Gamma (G, \mathcal A)$, we denote by $\lab (p)$ its label and by $p^{-1}$ the combinatorial inverse of $p$. We use the notation $d_{\mathcal A}$ and $|\cdot|_{\mathcal A}$ to denote the standard metric on $\Gamma (G,\mathcal A)$ and the length function on $G$ with respect to the (image of) $\mathcal A$.

In our paper, this terminology will be used in the following situation. Suppose that we have a group $G$, a collection of subgroups $\Hl$ of $G$, and a subset $X\subseteq G$ such that $X$ and the union of all $H_i$ together generate $G$. In this case we say that $X$ is a \emph{relative generating set} of $G$ with respect to $\Hl$. We think of $X$ and the subgroups $H_i$ as abstract sets and consider the disjoint unions
\begin{equation}\label{calA}
\mathcal H = \bigsqcup\limits_{i\in I} H_i\;\;\;\;\; {\rm and}\;\;\;\;\; \mathcal A= X \sqcup \mathcal H.
\end{equation}
Let $\mathcal A^\ast$ denote the free monoid on $\mathcal A$, i.e., the set of all words in the alphabet $\mathcal A$. For a word $W\in \mathcal A^\ast$, let $\|W\|$ denote its length. Given a word $a_1\ldots a_k\in \mathcal A^\ast$, we say that it \emph{represents} an element $g\in G$ if $g=a_1\cdots a_k$ in $G$.

\begin{conv}
Henceforth, we assume that all generating sets and relative generating sets are symmetric, i.e., closed under inversion.
\end{conv}

This convention implies that the alphabet $\mathcal A$ defined in (\ref{calA}) is also symmetric since so is $X$ and each $H_i$ is a subgroup. To take care of the possible ambiguity arising from the fact that distinct letters of $\mathcal A$ may represent the same element of $G$, we agree to think of $a^{-1}$ as a letter from $X$ (respectively, from $H_i$) if $a\in X$ (respectively, $a\in H_i$). Thus every element of $G$ can be represented by a word from $\mathcal A^\ast$.

In the settings described above, we can think of the Cayley graphs $\Gamma (H_i, H_i)$ as subgraphs of $\G$. For every $i\in I$, we introduce a (generalized) metric $\dl \colon H_i \times H_i \to [0, +\infty]$ as follows.

\begin{defn}
Given $g,h\in H_i$, let $\dl (g,h)$ be the length of a shortest path in $\G $ that connects $g$ to $h$ and contains no edges of $\Gamma (H_i, H_i)$. If no such a path exists, we set $\dl (h,k)=\infty $.
\end{defn}

Clearly $\dl $ satisfies the triangle inequality (with addition extended to $[0, +\infty]$ in the natural way). We are now ready to define hyperbolically embedded collections of subgroups introduced in \cite{DGO}.

\begin{defn}\label{hedefn}
A collection of subgroups $\Hl$ of $G$ is \emph{hyperbolically embedded  in $G$ with respect to a subset $X\subseteq G$}, denoted $\Hl \h (G,X)$, if the group $G$ is generated by the alphabet $\mathcal A$ defined by (\ref{calA}) and the following conditions hold.
\begin{enumerate}
\item[(a)] The Cayley graph $\G $ is hyperbolic.
\item[(b)] For every $n\in \mathbb N$ and every $i\in I$, the set $\left\{ h\in H_i\mid \dl(1,h)\le n\right\}$ is finite.
\end{enumerate}
Further, we say  that $\Hl$ is \emph{hyperbolically embedded} in $G$ and write $\Hl\h G$ if $\Hl\h (G,X)$ for some $X\subseteq G$.
\end{defn}

\begin{figure}
  \begin{center}
      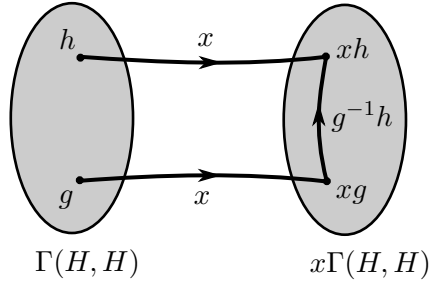
        \end{center}
   \caption{A path of length $3$ in the Cayley graph $\G$, where $G=\langle x\rangle\times H$ and $\mathcal A= \{ x\} \sqcup H$, connecting any elements $g,h\in H$ and avoiding edges of $\Gamma (H,H)$}\label{fig0}
 \end{figure}

To help the reader become familiar with these notions, we discuss three elementary examples.

\begin{ex}\label{bex}
\begin{enumerate}
\item[(a)] For any group $G$ we have $G\h G$.  Indeed we can take $X=\emptyset $ in this case. Then $\Gamma (G, \mathcal A)=\Gamma (G,G)$ and $\widehat\d_G (g,h)=\infty $ for any $g\ne h$.
\item[(b)] If $H$ is a finite subgroup of a group $G$, then $H\h G$. Indeed, it is straightforward to check that $H\h (G,X)$ for $X=G$.

\item[(c)] Let $G=H\times \mathbb Z$, let $X=\{ x\} $, where $x$ is a generator of $\mathbb Z$, and let $\mathcal A= \{ x\} \sqcup H$. It is easy to see that the graph $\G$ is quasi-isometric to a line and hence it is hyperbolic. However, every two elements $g,h\in H$ can be connected by a path of length at most $3$ in $\G$ that avoids edges of $\Gamma (H,H)$ (see Fig. \ref{fig0}). Thus $\widehat\d_H(g,h)\le 3$ and $H$ is not hyperbolically embedded in $G$ with respect to $X$ whenever $H$ is infinite.
\end{enumerate}
\end{ex}

Hyperbolically embedded collection of subgroups generalize peripheral structures of relatively hyperbolic groups.
Indeed, we have the following proposition. Readers unfamiliar with relative hyperbolicity can regard it as a definition.

\begin{prop}[{\cite[Proposition 4.28]{DGO}}]\label{rhhe}
A group $G$ is hyperbolic relative to a finite collection of subgroups $\Hl$ if and only if $\Hl\h(G, X)$ for some (equivalently, any) finite subset $X\subseteq G$.
\end{prop}

If $G$ is hyperbolic relative $\Hl$, the subgroups $H_i$ are called the \emph{peripheral subgroups} of $G$. For more on relative hyperbolicity, we refer to \cite{Osi06} and references therein.

We now turn to properties of hyperbolically embedded subgroups. An elaborated version of the argument from Example \ref{bex} (c) can be used to derive the following.

\begin{prop}[{\cite[Proposition 4.33]{DGO}}]\label{Prop:maln}
Let $G$ be a group, $\Hl $ a hyperbolically embedded collection of subgroups of $G$.
\begin{enumerate}
\item[(a)] For any $i\in I$ and any $g\in G\setminus H_i$, we have $|H_i \cap g^{-1}H_i g|<\infty $.
\item[(b)] For any distinct $i,j\in I$ and any $g\in G$, we have $|H_i \cap g^{-1}H_jg|<\infty$.
\end{enumerate}
\end{prop}

We mention one useful corollary. Recall that a group $G$ is said to have \emph{infinite conjugacy classes} (abbreviated \emph{ICC}) if the conjugacy class of every non-trivial element of $G$ is infinite.

\begin{cor}\label{Cor:ICC}
Let $G$ be a group containing two subgroups $H_1$, $H_2$ such that $\{ H_1, H_2\}\h G$ and $H_1\cap H_2=\{ 1\}$. Suppose that $S\le G$ and $S\cap H_i$ is infinite for $i=1,2$. Then, for every $g\in G\setminus \{1\}$, the set $g^S=\{ s^{-1}gs\mid s\in S\}$ is infinite. In particular, $S$ is  ICC.
\end{cor}

\begin{proof}
If $g^S$ is finite, then $g$ commutes with a finite index subgroup of $S$. Since both $S\cap H_1$ and $S\cap H_2$ are infinite, Proposition \ref{Prop:maln} implies that $g\in H_1\cap H_2=\{ 1\}$.
\end{proof}

The following result allows one to modify collections of hyperbolically embedded subgroups.

\begin{prop}[{\cite[Theorem 4.35]{DGO}}]\label{trans}
Let $G$ be a group, $\Hl$ a finite collection of subgroups of $G$. Suppose that $\Hl\h (G, X)$ for some $X\subseteq G$ and, for each $i\in I$, there is a collection of subgroups $\{ K_{ij}\}_{j\in J_i}$ of $H_i$ and a subset $Y_i\subseteq H_i$ such that $\{ K_{ij}\}_{j\in J_i}\h (H_i, Y_i)$. Then we have $\bigcup_{i\in I} \{ K_{ij}\}_{j\in J_i}\h (G,Z)$, where
\begin{equation}\label{Eq:ZX}
Z=X\cup \left(\bigcup_{i\in I}Y_i\right).
\end{equation}
\end{prop}

Note that the set $Z$ defined by (\ref{Eq:ZX}) is finite whenever so are $X$, $I$, and all $Y_i$. Combining this with Proposition \ref{rhhe}, we obtain the following.

\begin{cor}\label{Cor:rhrh}
Let $G$ be a group hyperbolic relative to a finite collection of subgroups $\Hl$. Suppose that each $H_i$ is hyperbolic relative to a finite collection of subgroups $\{ K_{ij}\}_{j\in J_i}$. Then $G$ is hyperbolic relative to $\bigcup_{i\in I} \{ K_{ij}\}_{j\in J_i}$.
\end{cor}

Further, let $G$ be a group generated by a finite set $X$. Note that the empty collection of subgroups is hyperbolically embedded in $G$ with respect to $X$ if and only if $G$ is a hyperbolic group. Thus, hyperbolic groups can always be excluded from finite hyperbolically embedded collections. More precisely, we have the following.

\begin{cor}\label{Cor:hypex}
Let $G$ be a group, $\Hl$, $\{ K_j \}_{j\in J}$ two finite collections of subgroups of $G$. Suppose that $\Hl\cup \{ K_j \}_{j\in J}\h (G, X)$ for some $X\subseteq G$. Assume also that each subgroup $H_i$ is hyperbolic and let $Y_i$ denote a finite generating set of $H_i$. Then $\{ K_j \}_{j\in J}\h (G,Z)$, where $Z$ is defined by (\ref{Eq:ZX}).
\end{cor}
\begin{proof}
We apply the proposition to the empty collection of subgroups of each $H_i$ and the subgroups $K_j\h (K_j, \emptyset)$ for $j\in J$.
\end{proof}

Using Proposition \ref{rhhe}, we obtain the following corollary for relatively hyperbolic groups (see \cite[Theorem 2.40]{Osi06} for a direct proof.)

\begin{cor}\label{Cor:rhh}
Let $G$ be a group hyperbolic relative to a finite collection of subgroups $\Hl \cup \{ K_j\}_{j\in J}$. If $H_i$ is hyperbolic for every $i\in I$, then $G$ is hyperbolic relative to $\{ K_j\}_{j\in J}$. In particular, a group hyperbolic relative to a finite collection of hyperbolic subgroups is itself hyperbolic.
\end{cor}

\subsection{Acylindrical hyperbolicity}\label{Sec:AH}

An isometric action of a group $G$ on a metric space $S$ is said to be {\it acylindrical} if, for every $\e>0$, there exist $R,N>0$ such that, for every two points $x,y\in S$ with $\d (x,y)\ge R$, there are at most $N$ elements $g\in G$ satisfying the inequalities
$$
\d(x,gx)\le \e \;\;\; {\rm and}\;\;\; \d(y,gy) \le \e.
$$
Informally, one can think of this condition as a kind of properness of the action on $S\times S$ minus a ``thick diagonal".

We begin with a classification of acylindrical group actions on hyperbolic spaces. Recall that an action of a group $G$ on a hyperbolic space $S$ is \emph{non-elementary} if the limit set of $G$ on the Gromov boundary $\partial S$ has infinitely many points.  The following classification of acylindrical actions is a simplification of {\cite[Theorem 1.1]{Osi16}}.

\begin{thm}\label{Thm:class}
Let $G$ be a group acting acylindrically on a hyperbolic space. Then $G$ satisfies exactly one of the following three conditions.
\begin{enumerate}
\item[(a)] $G$ has bounded orbits.
\item[(b)] $G$ has unbounded orbits and is virtually cyclic.
\item[(c)] The action of $G$ is non-elementary.
\end{enumerate}
\end{thm}

Thus, an acylindrical action of a group $G$ on a hyperbolic space is non-elementary if and only if $G$ is not virtually cyclic and has infinite orbits. Readers unfamiliar with the notions of Gromov's boundary and the limit set, can accept this as the definition of a non-elementary action.

The following theorem relates hyperbolically embedded subgroups to acylindrical actions. It can be easily derived from results of \cite{Osi16} as explained below.

\begin{thm}\label{Thm:AH}
Let $G$ be a group, $\Hl$ a finite collection of subgroups of $G$ such that $\Hl\h (G,X)$ for some $X\subseteq G$. Let also $\mathcal H$ and $\mathcal A$ be the alphabets defined by (\ref{calA}).
\begin{enumerate}
\item[(a)] There exists $Y\subseteq G$ such that $X\subseteq Y$, $\Hl\h (G,Y)$, and the action of $G$ on $\Gamma (G, \mathcal H \sqcup Y)$ is acylindrical.
\item[(b)] If $|X|<\infty$, then the action of $G$ on $\G$ is acylindrical.
\item[(c)] Suppose that the action of $G$ on $\G$ is acylindrical and there exists $i\in I$ such that $H_i$ is infinite and $H_i\ne G$. Then the action of $G$ on $\G$ is non-elementary.
\end{enumerate}
\end{thm}

\begin{proof}
We explain how to derive (a)--(c) from results in the literature. Part (a) is \cite[Theorem 5.4]{Osi16}. It is worth noting that enlarging the relative generating set in part (a) is necessary as the action of $G$ on $\G$ may not be acylindrical if $X$ is infinite (see the discussion before Theorem 5.4 in \cite{Osi16}).

By Proposition \ref{rhhe}, $G$ is hyperbolic relative to $\Hl$ if $|X|<\infty$. This allows us to apply \cite[Proposition 5.2]{Osi16} which gives us (b).

Finally, part (c) can be proved as follows. Suppose that $H_i$ is infinite and $H_i\ne G$ for some $i\in I$. Let $a\in G\setminus H_i$. By Proposition \ref{Prop:maln}, we have $|a^{-1}H_ia\cap H_i|<\infty$. The existence of such an element $a$ and the assumption $|H_i|=\infty $ allow us to apply \cite[Theorem 6.11]{DGO}, which implies that the group $G$ contains an element acting on $\G$ with unbounded orbits. Furthermore, the assumptions $|H_i|=\infty $ and $H_i\ne G$ imply that $G$ contains a non-cyclic free subgroup by \cite[Theorem 2.24 (c)]{DGO}). In particular, $G$ is not virtually cyclic. Therefore, the action of $G$ on $\G$ is non-elementary by Theorem \ref{Thm:class}.
\end{proof}

Every group has an acylindrical action on a hyperbolic space, namely the trivial action on the point. For this reason, we want to avoid elementary actions in the definition below.

\begin{defn}\label{ahdef}
A group $G$ is \emph{acylindrically hyperbolic} if admits a non-elementary acylindrical action on a hyperbolic space.
\end{defn}

We mention some equivalent characterizations.

\begin{thm}[{\cite[Theorem 1.2]{Osi16}}]\label{Thm:ah}
For any group $G$, the following conditions are equivalent.
\begin{enumerate}
\item[(A$_1$)] $G$ is acylindrically hyperbolic.
\item[(A$_2$)] There exists a (possibly infinite) generating set $X$ of $G$ such that the Cayley graph $\Gamma (G,X)$ is hyperbolic and the action of $G$ on $\Gamma (G,X)$ is non-elementary and acylindrical.
\item[(A$_3$)] $G$ contains a proper infinite hyperbolically embedded subgroup.
\end{enumerate}
\end{thm}

The class of acylindrically hyperbolic groups includes all non-elementary hyperbolic and relatively hyperbolic groups, mapping class groups of closed surfaces of non-zero genus, ${\rm Out}(F_n)$ for $n\ge 2$, non-virtually cyclic groups acting properly on proper $CAT(0)$ spaces and containing a rank-$1$ element, groups of deficiency at least $2$, most $3$-manifold groups, automorphism groups of some algebras (e.g., the Cremona group of birational transformations of the complex projective plane) and many other examples. For more details we refer to the survey \cite{Osi18}.

The next result will be used several times in our paper. It is a simplification of \cite[Lemma 7.1]{Osi16}.

\begin{lem}\label{Lem:NormAH}
Let $G$ be a group acting acylindrically and non-elementarily on a hyperbolic space $S$. For any infinite normal subgroup $N\lhd G$, the induced action of $N$ on $S$ is non-elementary; in particular, $N$ is acylindrically hyperbolic.
\end{lem}

Sizes of finite normal subgroups in every acylindrically hyperbolic group are uniformly bounded. Moreover, we have the following.

\begin{defn}\label{Def:FR}
Every acylindrically hyperbolic group contains a unique maximal finite normal subgroup (see the first part of \cite[Theorem 2.24]{DGO}). We denote it by $K(G)$ and call it the \emph{finite radical} of $G$.
\end{defn}

\begin{thm}[{\cite[Theorem 2.35]{DGO}}]\label{Thm:HypICC}
An acylindrically hyperbolic group $G$ is ICC if and only if $K(G)=\{ 1\}$.
\end{thm}

We conclude this section with the definition of a loxodromic element and examples of hyperbolically embedded subgroups, which will be used many times in our paper.

\begin{defn}\label{Def:lox}
An element $g$ of a group $G$ acting on a hyperbolic space $S$ is called \emph{loxodromic} if it acts as a translation along a bi-infinite quasi-geodesic in $S$. If the action of $G$ on $S$ is acylindrical, this is equivalent to the requirement that $\langle g\rangle$ has unbounded orbits (see \cite[Lemma 2.2]{Bow}).
\end{defn}

If $G$ is hyperbolic relative to a finite collection $\Hl$ and $X$ is a finite relative generating set of $G$ with respect to $\Hl$, then the action of $G$ on $\Gamma (G, \mathcal A)$ (where $\mathcal A$ is defined by (\ref{calA})) is acylindrical by Proposition \ref{rhhe} and part (b) of Theorem \ref{Thm:AH}. In these settings, an element $g\in G$ is loxodromic if and only if $g$ has infinite order and is not conjugate to an element of one of the peripheral subgroups (see \cite[Theorem 4.23]{Osi06}). Thanks to this characterization, whether or not an element of a relatively hyperbolic group $G$ is loxodromic is independent of the choice of a particular finite relative generating set of $G$. Thus we can simply talk about loxodromic elements of a relatively hyperbolic group if the peripheral collection is understood.

If $G$ is hyperbolic, it is hyperbolic relative to the trivial subgroup and every infinite order element is loxodromic with respect to this collection. For more details on loxodromic elements, we refer the reader to Section 6 of \cite{Osi16}.

The following is proved in \cite[Lemma 6.5]{DGO}.

\begin{lem}\label{Lem:Eg}
Suppose that a group $G$ acts acylindrically on a hyperbolic space $S$. Then every loxodromic element $g\in G$ is contained in a unique maximal virtually cyclic subgroup of $G$.
\end{lem}

\begin{defn}
In the settings of Lemma \ref{Lem:Eg}, we denote the unique maximal virtually cyclic subgroup of $G$ containing $g$ by $E(g)$.
\end{defn}

Recall also that two elements $a$, $b$ of a group $G$ are said to be \emph{commensurable} if some non-zero powers of $a$ and $b$ are conjugate in $G$. The following theorem was proved in \cite[Corollary 3.12]{AMS} (it is worth noting that the WPD condition assumed there follows immediately from acylindricity of the action). Similar but less precise statements were also proved in \cite{DGO,Hull}.

\begin{thm}[Antolin--Minasyan--Sisto]\label{Thm:Eg}
Let $G$ be a group, $\Hl$ a finite collection of subgroups of $G$ such that $\Hl\h (G,X)$ for some $X\subseteq G$. Let also $\mathcal A$ be the alphabet defined by (\ref{calA}). Suppose that the action of $G$ on $\G$ is acylindrical. Then for any collection of pairwise non-commensurable loxodromic (with respect to the action on $\Gamma (G,A)$) elements $g_1, \ldots, g_n\in G$, we have $\{ E(g_1), \ldots , E(g_n)\}\cup \Hl \h (G, X)$.
\end{thm}

Combining Theorem \ref{Thm:Eg} with Proposition \ref{rhhe}, Theorem \ref{Thm:AH} (b), Lemma \ref{Lem:Eg}, and the characterization of loxodromic elements of relatively hyperbolic groups discussed above, we obtain the following particular case. For a direct proof see \cite{Osi06b}.

\begin{cor} \label{Cor:EgRH}
Let $G$ be a group hyperbolic relative to a collection of subgroups $\Hl$. Suppose that $g_1, \ldots, g_n\in G$ are pairwise non-commensurable elements of infinite order that are not conjugate to elements of the peripheral subgroups. Then every $g_i$ is contained in a unique maximal virtually cyclic subgroup $E(g_i)$ of $G$ and $G$ is hyperbolic relative to $\Hl\cup \{ E(g_1), \ldots, E(g_n)\} $.
\end{cor}

\section{Small cancellation in groups with hyperbolically embedded subgroups}

\subsection{Olshanskii's small cancellation conditions}\label{Sec:Ols}

We begin by reviewing the necessary background from \cite{Ols93}. Let $G$ be a group generated by an alphabet $\mathcal A$. Consider any presentation
\begin{equation}\label{GA0}
G=\langle \mathcal A\; | \; \mathcal O\rangle .
\end{equation}
We do not assume that this presentation is finite; in particular, we can take $\mathcal O$ to be the set of all words in $\mathcal A$ representing the identity in $G$. We write $W_1\equiv W_2$ to express letter-for-letter equality of two words $W_1,W_2\in \mathcal A^\ast$.

In what follows, we use the standard terms \emph{vertices}, \emph{edges}, and \emph{faces} for \emph{$0$-cells}, \emph{$1$-cells}, and \emph{$2$-cells} of planar $CW$-complexes. Recall that a {\it van Kampen diagram} $\Delta $ over (\ref{GA0})
is a finite, oriented, connected, simply-connected, planar $CW$-complex endowed with a labelling
function $\lab\colon E(\Delta )\to \mathcal A$, where $E(\Delta)$ denotes the set of edges of $\Delta$, such that $\lab
(e^{-1})\equiv (\lab (e))^{-1}$ for all $e\in E(\Delta)$.

Given a face $\Pi $ of $\Delta $, we denote by $\partial \Pi$ the boundary of $\Pi $; similarly, $\partial \Delta $ denotes the boundary of $\Delta $. The labels of $\partial \Pi $ and $\partial \Delta $ are defined up to a cyclic permutation. An additional requirement is that for any face $\Pi $ of $\Delta $ the boundary label $\lab (\partial \Pi)$ is equal to (a cyclic permutation of) a word $P^{\pm 1}$, where $P\in \mathcal O$.

By the van Kampen lemma, a word $w$ in the alphabet $\mathcal A$ represents the identity in the group $G$ given by
(\ref{GA0}) if and only if there exists a van Kampen diagram $\Delta $ over (\ref{GA0}) such that $\lab(\partial \Delta)\equiv W$.

Given a set of words $\mathcal R$ in the alphabet $\mathcal A$, we consider the quotient group of $G$
given by
\begin{equation}\label{quot}
\overline{G} = G/\ll \mathcal R\rr = \langle \mathcal A\; |\; \mathcal O\cup \mathcal R\rangle .
\end{equation}

To analyze properties of the quotient group $\overline{G}$ we will make use of a simplified version of the $C(\mu, \e, \lambda, c, \rho)$ and $C_1(\mu, \e, \lambda, c, \rho)$ small cancellation conditions introduced in \cite{Ols93}. The difference is that we require the relators to be geodesic rather than $(\lambda, c)$-quasi-geodesic. Thus, the conditions $C(\mu, \e, \rho)$ and $C_1(\mu, \e, \rho)$ correspond to $C(\mu, \e, 1,0, \rho)$ and $C_1(\mu, \e, 1,0, \rho)$, respectively, in the notation used in \cite{Ols93} and the subsequent papers \cite{Hull,HO,Osi10}.

More precisely, we say that a word $W$ in the alphabet $\mathcal A$ is {\it $(G, \mathcal A)$-geodesic} if any path in the Cayley graph $\Gamma (G, \mathcal A)$ labeled by $W$ is geodesic. Recall that $\mathcal R$ is said to be {\it symmetric} if for any $R\in \mathcal R$, $\mathcal R$ contains all cyclic shifts of $R^{\pm 1}$. If a word $W\in \mathcal A^\ast$ decomposes as $W\equiv UV$ for some $U,V\in \mathcal A^\ast$, we say that $U$ is an \emph{initial subword} of $W$.  The length of a word $W\in \mathcal A^\ast$ (i.e., the number of letters in $W$) is denoted by $\| W\|$.

\begin{figure}
   \begin{center}
   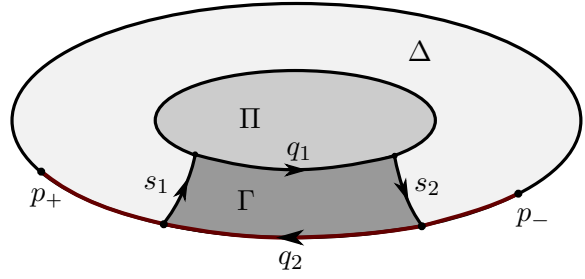
   \end{center}
   \vspace*{-3mm}
   \caption{A contiguity subdiagram}\label{fig1}
 \end{figure}

\begin{defn}\label{DefSC}
A symmetric set of words $\mathcal R$ satisfies the {\it $C(\e , \mu , \rho )$--condition} for some $\e \ge 0$ and $\mu, \rho >0$,
if the following conditions hold.
\begin{enumerate}
\item[(a)] All words in $\mathcal R$ are $(G, \mathcal A)$-geodesic and have length at least  $\rho $.
\item[(b)] Suppose that words $R,R^\prime \in \mathcal R$ have initial subwords $U$ and $U^\prime$, respectively,  such that
\begin{equation}\label{piece1}
\max \{ \| U\| ,\, \| U^\prime\| \} \ge \mu \min\{\| R\|, \| R^\prime\|\}
\end{equation}
and $U^\prime = YUZ$ in $G$ for some words $Y$, $Z$ of length
\begin{equation}\label{piece2}
\max \{ \| Y\| , \,\| Z\| \} \le \e.
\end{equation}
Then $YRY^{-1}=R^\prime$ in $G$.
\end{enumerate}
Further, we say that $\mathcal R$ satisfies the {\it $C_1(\e , \mu , \rho )$--condition} if,  in addition to (a) and (b), we have the following.
\begin{enumerate}
\item[(c)] Suppose that a word $R\in \mathcal R$ contains two disjoint subwords $U$ and $U^\prime$ such that $U^\prime = YUZ$ or $U^\prime = YU^{-1}Z$ in $G$ for some words $Y$, $Z$ and the inequality (\ref{piece2}) holds. Then
    $$ \max \{ \| U\| ,\, \| U^\prime\| \}< \mu \| R\|.$$
\end{enumerate}
\end{defn}

Let $\Delta $ be a van Kampen diagram over (\ref{quot}). A face in $\Delta$ is called an {\it $\mathcal R$-face} if its boundary label is a word from $\mathcal R$. Suppose that $p$ is a subpath of $\partial\Delta$ and there is a simple closed path
\begin{equation}\label{pathp}
c=s_1q_1s_2q_2
\end{equation}
in $\Delta $, where $q_1$ is a subpath of
$\partial \Pi $, $q_2$ is a subpath of $p$, and
\begin{equation}\label{sides}
\max \{\ell( s_1),\, \ell(s_2) \} \le \e
\end{equation}
for some constant $\e >0$. Let $\Gamma $ denote the subdiagram of
$\Delta $ bounded by $c$. If $\Gamma $
contains no $\mathcal R$-faces, we say that $\Gamma $ is an {\it $\e $--contiguity subdiagram} of $\Pi $ to the segment $p$ of $\partial \Delta $. The ratio $\ell(q_1)/\ell(\partial \Pi )$ is called the {\it contiguity degree} of $\Pi $ to $\partial \Delta$ and is denoted by $(\Pi , \Gamma , \partial \Delta)$. Since $\Gamma$ contains no $\mathcal R$-faces, it can be considered a diagram over (\ref{GA0}).

The analogue of the classical Greendlinger lemma for small cancellation quotients of hyperbolic groups was proved by Olshanskii in \cite[Lemma 6.6]{Ols93}. As observed in \cite[Lemma 4.4]{Osi10}, the same proof works in the more general context of groups with hyperbolic (but not necessarily locally finite) Cayley graphs. Below we reproduce a simplified version of the latter result replacing the more general $C( \e, \mu, \lambda, c,\rho)$ condition considered in \cite{Ols93,Osi10} with the stronger condition $C(\e, \mu, \rho)$.

\begin{lem}[Olshanskii]\label{GOL}
Suppose that the Cayley graph $\Gamma (G, \mathcal A)$ of a group $G$ given by the presentation (\ref{GA0}) is hyperbolic. Then for every $\mu \in (0, 1/16]$, there exist $\e, \rho >0$ such that the following holds.

Let $\mathcal R$ be a symmetric set of words in the alphabet $\mathcal A$ satisfying $C(\e, \mu , \rho )$. Assume that $W\equiv W_1W_2W_3W_4$ is a word in the alphabet $\mathcal A$ representing $1$ in $\overline{G}=G/\ll \mathcal R\rr$ such that each of the words $W_1, \ldots, W_4$ is $(G, \mathcal A)$-geodesic and $W\ne 1$ in $G$. Then there exist a van Kampen diagram $\Delta $ over (\ref{quot}) with boundary $\partial \Delta = p_1p_2p_3p_4$, where $\lab(p_i)\equiv W_i$ for $i=1, \ldots, 4$,  an $\mathcal R$-face $\Pi $ of $\Delta $, and $\e$-contiguity subdiagrams $\Gamma_i$ of $\Pi $ to the segments $p_i$ of $\partial \Delta $ (see Fig. \ref{fig4cs}) such that $$\sum_{i=1}^4(\Pi, \Gamma_i, p_i)\ge 1-13\mu .$$
\end{lem}

\begin{figure}
   \begin{center}
   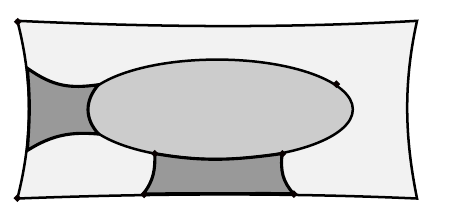
   \end{center}
   \vspace*{-3mm}
   \caption{}\label{fig4cs}
 \end{figure}

\begin{rem}
The lemma can be proved for the concatenation of $n$ geodesic words for any $n\in \NN$. However, the case $n=4$ is sufficient for most applications.
\end{rem}

\subsection{Centralizers in small cancellation quotients}

In this section, we study centralizers of elements in small cancellation quotients of groups acting on hyperbolic spaces. Our main result is Lemma \ref{Lem:HO} below. It plays an important role in the proofs of Theorem \ref{app1} and Theorem \ref{app2}. We begin with an auxiliary result.

\begin{lem}\label{Lem:HO}
Let $G$ be a group given by presentation (\ref{GA0}) such that the Cayley graph $\Gamma (G, \mathcal A)$ is $\delta$-hyperbolic for some $\delta \ge 0$. For any $\e\ge 4\delta$, any $\mu, \rho>0$, and any set of words $\mathcal R$ satisfying the $C_1(\e, \mu ,\rho )$ condition, the following holds.

Let $U$ and $V$ be disjoint subwords of some $R\in \mathcal R$, and let $p$ be a geodesic path in $\Gamma(G,\mathcal A)$. Suppose that $a_1b_1c_1d_1$ and $a_2b_2c_2d_2$ are quadrilaterals in $\Gamma(G, \mathcal A)$ such that $\lab(b_1)\equiv U$, $\lab(b_2)\equiv V^{\pm 1}$, $d_1$, $d_2$ are subpaths of $p$, and $\ell(a_i), \ell(c_i)\leq\e$ for $i=1,2$. Then the overlap between $d_1$ and $d_2$ has length less than $\mu\|R\|+10\e$.
\end{lem}

\begin{proof}
Let $o$ denote the overlap of $d_1$ and $d_2$ (see Fig. \ref{figC1}). Assume that $\ell(o)\ge \mu\|R\|+10\e$.

By Lemma \ref{Lem:HRect} (b), we can choose points $x_1$, $y_1$ on $b_1$, such that
\begin{equation}\label{do-y1}
\max\{\d(o_-, y_1), \d(o_+, x_1)\}\leq 2\delta + \max\{\ell(a_1), \ell(c_1)\} \le 2\delta +\e \le 1.5\e.
\end{equation}
Similarly, we choose $x_2$ and $y_2$ on $b_2$ satisfying $\d(o_+, x_2) \le 1.5\e$ and $\d(o_-, y_2)\leq 1.5\e$. (If the quadrilaterals $a_1b_1c_1d_1$ and $a_2b_2c_2d_2$ are positioned as shown on Fig. \ref{figC1}, we can take $y_1=(c_1)_-$, $x_2=(a_2)_+$, and improve the upper bounds for $\d(o_-, y_1)$ and $\d(o_+, x_2)$ to $\e$; however, we choose to make our proof independent of the relative position of the quadrilaterals at the cost of a slight increase of the constants.)

\begin{figure}
   \begin{center}
   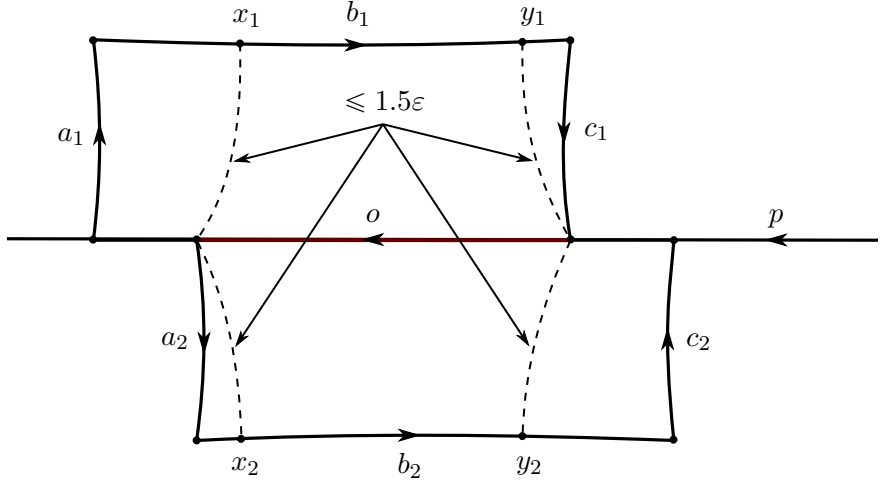
   \end{center}
   \vspace*{-3mm}
   \caption{The proof of Lemma \ref{Lem:HO}}\label{figC1}
 \end{figure}

By the triangle inequality, we have $\d (x_1, x_2)\le 3\e $ and $\d (y_1, y_2)\le 3\e$. Since $p$ is geodesic, so is $o$. Using (\ref{do-y1}), we obtain
$$
\d(x_1, y_1) \ge \ell(o) - \d(o_+, x_1) -\d(o_-, y_1) \ge \mu \| R\| + 7\e \ge \mu \| R\|+\d (x_1, x_2)+\d (y_1, y_2) +4\delta.
$$
Applying Lemma \ref{Lem:HRect} (c) to the geodesic quadrilateral with the consecutive vertices $x_1$, $y_1$, $y_2$, $x_2$, we obtain a subpath $t$ of $b_1^{\pm 1}$ such that $\d (t_\pm, b_2)\le 2\delta \le \e$ and $$\ell (t) \ge \d(x_1, y_1) - \d (x_1, x_2)-\d (y_1, y_2) -4\delta \ge \mu \| R\|.$$ Since $\lab(b_1)\equiv U$ and $\lab(b_2)\equiv V^{\pm 1}$ are disjoint subwords of $R$, this contradicts the $C_1(\e, \rho, \mu)$ small cancellation condition.
\end{proof}

\begin{lem}\label{Prop:HO}
Let $G$ be a group given by the presentation (\ref{GA0}) such that the Cayley graph $\Gamma (G, \mathcal A)$ is hyperbolic. For any $N\ge 0$, there exist $\e,\rho>0$ such that the following holds.

Let $\mathcal R$ be a set of words in the alphabet $\mathcal A$ satisfying the $C_1(\e, 1/100, \rho)$ small cancellation condition, $\overline G$ the quotient group defined by (\ref{quot}). For any $g\in G$ of length $|g|_\mathcal A\le N$, we have $C_{\overline{G}}(\gamma(g))=\gamma(C_G(g))$, where $\gamma\colon G\to \overline{G}$ is the natural homomorphism.
\end{lem}

\begin{proof}
For hyperbolic groups, the lemma was proved in \cite{Ols93}. For relatively hyperbolic groups, this result can be derived from
\cite[Lemma 5.5]{HO}. Our proof is based on the same idea as the proof of \cite[Lemma 5.5]{HO}.

Fix some $N\in \NN$ and let $\mu=1/100$. Note that the $C(\e, \mu, \rho)$ condition becomes stronger as $\e$ and $\rho$ increase. Therefore, we can choose a sufficiently large value of $\e$ and then a sufficiently large value of $\rho$ such that the conclusions of Lemma \ref{GOL} and Lemma \ref{Lem:HO} simultaneously hold and, in addition,
\begin{equation}\label{Eq:rho}
\rho \ge 100(2\e+N).
\end{equation}
In particular, we have $\e/\rho <1/200$.

Let $g\in G$ be an element of length $|g|_\mathcal A\le N$ and let $U$ be a $(G, \mathcal A)$-geodesic word representing $g$ in $G$. The inclusion $\gamma(C_G(g))\le C_{\overline G} (\gamma(g))$ is trivial, we only need to prove the inclusion $C_{\overline G} (\gamma(g))\le \gamma(C_G(g))$. Let $c\in C_{\overline G}(\gamma (g))$ and let $V$ be a $(\overline G, \mathcal A)$-geodesic word representing $c$ in $\overline G$ (in particular, $V$ is $(G, \mathcal A)$-geodesic). We will show that
\begin{equation}\label{Eq:c}
c\in \gamma(C_G(g)).
\end{equation}

Clearly, the word $UVU^{-1}V^{-1}$ represents $1$ in $\overline G$. To prove the lemma, it suffices to show that $UVU^{-1}V^{-1}=1$ in $G$. Indeed, then $V$ represents an element of $C_G(g)$ and we obtain (\ref{Eq:c}).

Arguing by contradiction, assume that $UVU^{-1}V^{-1}\ne 1$ in $G$. By Lemma \ref{GOL}, there exist a van Kampen diagram $\Delta $ over (\ref{quot}) whose boundary decomposes as $\partial \Delta = p_1q_1p_2^{-1}q_2^{-1}$, where $$\lab(p_1)\equiv \lab(p_2)\equiv U,\;\;\; \lab(q_1)\equiv \lab(q_2)\equiv V,$$ an $\mathcal R$-face $\Pi $ of $\Delta $, and $\e$-contiguity subdiagrams $\Gamma_1$, $\Gamma _2$ (respectively, $\Sigma _1, \Sigma_2$) of $\Pi $ to the segments $p_1$, $p_2$ (respectively, $q_1$, $q_2$)  of $\partial \Delta $ such that
\begin{equation}\label{Eq:contdeg}
\sum_{i=1}^2(\Pi, \Gamma_i, p_i)+ \sum_{i=1}^2(\Pi, \Sigma_i, q_i)\ge 1-13\mu= 0.87
\end{equation}
(see Fig. \ref{figCentr} (a)).

\begin{figure}
   \begin{center}
   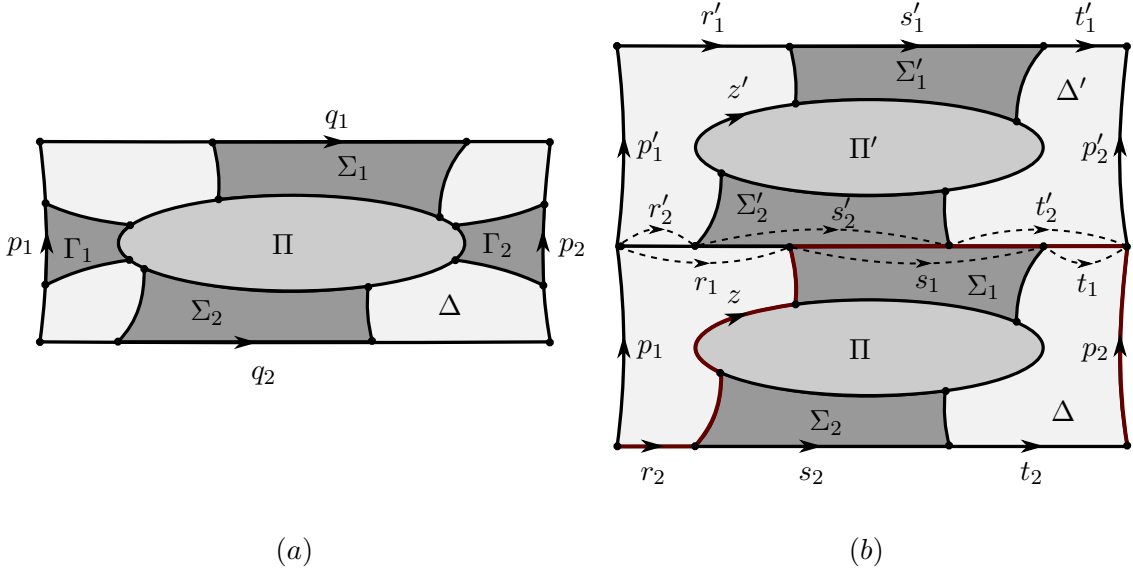
   \end{center}
   \vspace*{-3mm}
   \caption{The diagrams $\Delta $ and $\Omega$.}\label{figCentr}
 \end{figure}

Since $\lab(\partial\Pi)$ is $(G, \mathcal A)$-geodesic, the length of the common segment of $\Gamma _1$ and $\Pi $ is bounded from above by the sum of lengths of the other three sides of $\Gamma_1$, which is at most $2\e +\ell (p_1)\le 2\e +N$. Hence, $(\Pi, \Gamma_1, p_1)\le  (2\e+N)/\rho \le 1/100 $ by (\ref{Eq:rho}). Similarly, we have $(\Pi, \Gamma_2, p_2)\le  1/100 $. Combining these inequalities with (\ref{Eq:contdeg}), we obtain
\begin{equation}\label{Eq:contdeg1}
\sum_{i=1}^2(\Pi, \Sigma_i, q_i)\ge 0.85.
\end{equation}

For $i=1,2$, let $q_i= r_is_it_i$, where $s_i$ is the common segment of $\Sigma _i$ and $q_i$. We claim that
\begin{equation}\label{Eq:lsi}
\ell(s_i) \ge 0.3|\partial \Pi|
\end{equation}
for $i=1,2$. Indeed, suppose that $\ell (s_1)< 0.3|\partial \Pi|$. Since $\lab(\partial \Pi)$ is $(G, \mathcal A)$-geodesic and $|\partial \Pi|\ge \rho \ge 200\e$ by (\ref{Eq:rho}), we have $(\Pi, \Sigma_1, q_1) <  (0.3 |\partial \Pi| + 2\e )/|\partial \Pi|\le  0.31$. Using (\ref{Eq:contdeg1}), we obtain $(\Pi, \Sigma_2, q_2)\ge 0.85 - (\Pi, \Sigma_1, q_1) > 0.54$. Since $V$ is $(\overline G, \mathcal A)$-geodesic, the path $q_2$ is geodesic in $\Delta$; therefore, so is $s_2$. Comparing the length of $s_2$ to the path going along the ``$\e$-sides" of $\Sigma_2$ and around the face $\Pi$, we obtain
$$
\ell(s_2)\le 2\e + |\partial \Pi| (1-(\Pi, \Sigma_2, q_2))\le 2\e + 0.46|\partial \Pi| \le 0.47 |\partial \Pi|.
$$
Recall that the sides of $\Delta $ are $(G,\mathcal A)$-geodesic. Applying the triangle inequality to the sides of $\Sigma_2$, we obtain
$$
\ell(s_2) \ge |\partial \Pi|(\Pi, \Sigma_2, q_2) -2\e > 0.54 |\partial \Pi| -2\e \ge 0.53 |\partial \Pi|,
$$
which contradicts the previous inequality. Thus, $\ell(s_1) \ge 0.3|\partial \Pi|$ and similarly $\ell(s_2) \ge 0.3|\partial \Pi|$.

Let $\Delta ^\prime$ denote a copy of $\Delta$; we use primes to distinguish between parts of $\Delta ^\prime$ and the corresponding to parts of $\Delta$. E.g., $p_1^\prime$ will denote the path in $\Delta^\prime$ corresponding to the path $p_1$ in $\Delta$, $\Sigma_1^\prime$ will denote the subdiagram of in $\Delta^\prime$ corresponding to $\Sigma_1$, etc. Consider the van Kampen diagram $\Omega$ formed by gluing $\Delta $ and $\Delta ^\prime$ along $q_1$ and $q_2^\prime$ (see Fig. \ref{figCentr} (b)). For definiteness, we will assume that
\begin{equation}\label{Eq:r1r2}
\ell (r_2)\le  \ell (r_1)
\end{equation}
(the other case is similar). By Lemma \ref{Lem:HO}, the overlap between $s_1$ and $s_2^\prime$ in $\Omega$ is at most $0.01 |\partial \Pi| + 10 \e \le  0.06 |\partial \Pi|$. Using (\ref{Eq:contdeg1}), (\ref{Eq:r1r2}), and (\ref{Eq:lsi}), we obtain that the length of the path $z$ highlighted in red on Fig. \ref{figCentr} (b) satisfies
$$
\begin{array}{rcl}
\ell(z) & \le & \ell(r_2) + \e + 0.15 |\partial \Pi| +\e +\ell(s_1t_1)+\ell(p_2)\le \\ &&\\ &&
 2\e + 0.15 |\partial \Pi| + \ell(q_2) - (\ell(s_2) - 0.06 |\partial \Pi|) +\ell(p_2)\le \\ &&\\ &&
 2\e +0.15 |\partial \Pi|+ \ell(q_2) - 0.24 |\partial \Pi| +N \le \\ &&\\ &&
 \ell(q_2) +2\e+N- 0.09 \rho < \ell(q_2).
\end{array}
$$
This contradicts the assumption that $V$ is $(\overline{G}, \mathcal A)$-geodesic since $\lab(z)$ and $\lab(q_2) \equiv V$ represent the same element of $\overline G$.
\end{proof}

\subsection{Attracting geodesics}\label{Sec:AG}
\begin{figure}
   \begin{center}
\begingroup%
  \makeatletter%
  \providecommand\color[2][]{%
    \errmessage{(Inkscape) Color is used for the text in Inkscape, but the package 'color.sty' is not loaded}%
    \renewcommand\color[2][]{}%
  }%
  \providecommand\transparent[1]{%
    \errmessage{(Inkscape) Transparency is used (non-zero) for the text in Inkscape, but the package 'transparent.sty' is not loaded}%
    \renewcommand\transparent[1]{}%
  }%
  \providecommand\rotatebox[2]{#2}%
  \newcommand*\fsize{\dimexpr\f@size pt\relax}%
  \newcommand*\lineheight[1]{\fontsize{\fsize}{#1\fsize}\selectfont}%
  \ifx\svgwidth\undefined%
    \setlength{\unitlength}{280.41985701bp}%
    \ifx\svgscale\undefined%
      \relax%
    \else%
      \setlength{\unitlength}{\unitlength * \real{\svgscale}}%
    \fi%
  \else%
    \setlength{\unitlength}{\svgwidth}%
  \fi%
  \global\let\svgwidth\undefined%
  \global\let\svgscale\undefined%
  \makeatother%
  \begin{picture}(1,0.27800065)%
    \lineheight{1}%
    \setlength\tabcolsep{0pt}%
    \put(0,0){\includegraphics[width=\unitlength,page=1]{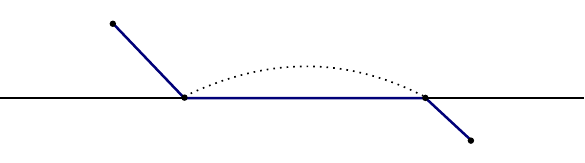}}%
    \put(0.18147723,0.25604696){\color[rgb]{0,0,0}\makebox(0,0)[lt]{\lineheight{1.25}\smash{\begin{tabular}[t]{l}$x$\end{tabular}}}}%
    \put(0.81372914,0.00502873){\color[rgb]{0,0,0}\makebox(0,0)[lt]{\lineheight{1.25}\smash{\begin{tabular}[t]{l}$b$\end{tabular}}}}%
    \put(0.28602336,0.18534311){\color[rgb]{0,0,0}\makebox(0,0)[lt]{\lineheight{1.25}\smash{\begin{tabular}[t]{l}$\ell(r)\ge \d(x,y)-\d(x, p)-\d(y, p)-c$\end{tabular}}}}%
    \put(0,0){\includegraphics[width=\unitlength,page=2]{fig2.pdf}}%
    \put(0.51636097,0.06541705){\color[rgb]{0,0,0}\makebox(0,0)[lt]{\lineheight{1.25}\smash{\begin{tabular}[t]{l}$r$\end{tabular}}}}%
    \put(0.06575839,0.07505561){\color[rgb]{0,0,0}\makebox(0,0)[lt]{\lineheight{1.25}\smash{\begin{tabular}[t]{l}$p$\end{tabular}}}}%
  \end{picture}%
\endgroup%

   \end{center}
   \vspace{-5mm}\caption{A $c$-attracting geodesic}\label{fig2}
\end{figure}

In this section we introduce the notion of an attracting geodesic, which is crucial for the proof of Proposition \ref{Prop:GL}. By abuse of terminology, we identify a paths $p\colon [0,1]\to S$ in a metric space $S$ with its image. Given a path $p$, we denote by $p_-$, $p_+$ its origin and the terminal point, respectively. The inverse path is denoted by $p^{-1}$. By $\d(s,p)$ we denote the distance from a point $s\in S$ to the path $p$ defined in the usual way. We say that two paths $p$ and $q$ \emph{share an unoriented segment} $r$ if $p=arb$ and $q=cr^{\pm 1}d$.

\begin{defn}
Let $\Gamma $ be a graph. We say that a (combinatorial) geodesic path $p$ in $\Gamma $ is \emph{$c$-attracting} for some $c\in [0, \infty)$ if for any vertices $x$, $y$ of $\Gamma $ such that $\d(x,y) - \d (x, p)- \d(y,p)\ge c$, every geodesic in $\Gamma $ going from $x$ to $y$ shares an unoriented segment of length at least $\d(x,y) - \d (x, p)- \d(y,p)-c$ with $p$. (By definition, a segment of length $0$ is a point.)
\end{defn}

This definition is motivated by the obvious fact that in a tree, all geodesics are $0$-attracting. On the other hand, in hyperbolic graphs and even in quasi-trees, long $c$-attracting geodesics may not exist (note that every geodesic of length less than $c$ in any graph is $c$-attracting though).

Our next goal is to construct arbitrarily long $2$-attracting geodesics in Cayley graphs associated to hyperbolically embedded collections of subgroups. Throughout the rest of this section, we make the following assumptions.

\noindent($\ast$)\;\;\; \emph{$G$ is a group, $\Hl$ is a collection of subgroups of $G$, $X$ is a relative generating set of $G$ with respect to $\Hl$. }

\noindent We keep the notation $\mathcal A$ and $\widehat\d_{H_i}$ introduced in Section \ref{Sec:Hyp} (see (\ref{calA})). In addition, we assume that

\noindent($\ast\ast$)\;\;\; \emph{$\G$ is hyperbolic.}

\noindent  Note that, in this section, we do not require $\Hl$ to be hyperbolically embedded in $G$ with respect to $X$.

We make use of the following terminology introduced in \cite{DGO} (in the particular case of relatively hyperbolic groups it goes back to \cite{Osi06}).

\begin{defn}
Let $p$ be a path in the Cayley graph $\G $. A subpath $q$ of $p$ is called an {\it $H_i $-component} (or simply a \emph{component}) of $p$ for some $i\in I $, if $q$ has at least one edge, the label of $q$ is a word in the alphabet $H_i\subseteq \mathcal A$, and $q$ is not contained in any longer subpath of $p$ with these properties. Two $H_i$-components $q_1, q_2$ of $p$ are {\it connected} if all vertices of $q_1$ and $q_2$ belong to the same coset of $H_i$ in $G$. An $H_i$-component $q$ of a path $p$ is called {\it isolated } if no distinct $H_i$-component of $p$ is connected to $q$.
\end{defn}

Geometrically, the fact that $H_i$-components $q_1, q_2$ of $p$ are connected means that any two vertices $u$ and $v$ of $q_1$ and $q_2$, respectively, can be connected by an edge of $\G$ labelled by an element of $H_i$.

\begin{rem}
If $p$ is a geodesic in $\G$, then every $H_i$-component of $p$ is isolated in $p$ and consists of a single edge. This observation will be used many times in this section.
\end{rem}

By a \emph{geodesic $n$-gon} in a metric space we mean a loop consisting of $n$ geodesic segments. The lemma below is a simplification of \cite[Proposition 4.14]{DGO}.

\begin{lem}\label{Omega}
Under the assumptions ($\ast$) and ($\ast\ast$), there exists a constant $D$ satisfying the following condition. Let $p$ be a geodesic $n$-gon in $\G $. Then for every $i\in I$ and every isolated $H_i$-component $q$ of $p$, the element $h\in H_i$ represented by the label of $q$  satisfies $\dl (1, h)\le Dn$.
\end{lem}

We will need the following result, which is an improved version of \cite[Lemma 4.21]{DGO}. Recall that the length of a word $u \in \mathcal A^\ast$ is denoted by $\| u\|$.

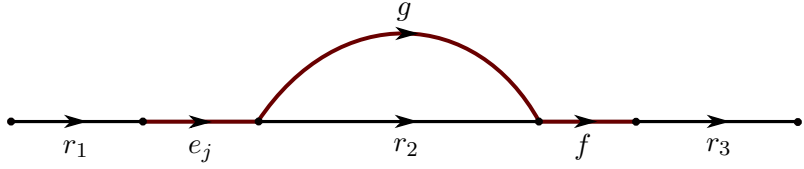
\begin{figure}
   \begin{center}
\begingroup%
  \makeatletter%
  \providecommand\color[2][]{%
    \errmessage{(Inkscape) Color is used for the text in Inkscape, but the package 'color.sty' is not loaded}%
    \renewcommand\color[2][]{}%
  }%
  \providecommand\transparent[1]{%
    \errmessage{(Inkscape) Transparency is used (non-zero) for the text in Inkscape, but the package 'transparent.sty' is not loaded}%
    \renewcommand\transparent[1]{}%
  }%
  \providecommand\rotatebox[2]{#2}%
  \newcommand*\fsize{\dimexpr\f@size pt\relax}%
  \newcommand*\lineheight[1]{\fontsize{\fsize}{#1\fsize}\selectfont}%
  \ifx\svgwidth\undefined%
    \setlength{\unitlength}{299.48691245bp}%
    \ifx\svgscale\undefined%
      \relax%
    \else%
      \setlength{\unitlength}{\unitlength * \real{\svgscale}}%
    \fi%
  \else%
    \setlength{\unitlength}{\svgwidth}%
  \fi%
  \global\let\svgwidth\undefined%
  \global\let\svgscale\undefined%
  \makeatother%
  \begin{picture}(1,0.19991142)%
    \lineheight{1}%
    \setlength\tabcolsep{0pt}%
    \put(0,0){\includegraphics[width=\unitlength,page=1]{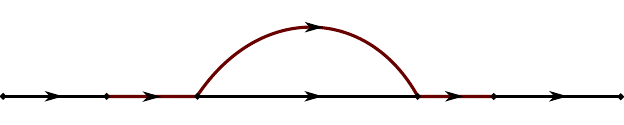}}%
    \put(0.07028398,0.00641246){\color[rgb]{0,0,0}\makebox(0,0)[lt]{\lineheight{1.25}\smash{\begin{tabular}[t]{l}$r_1$\end{tabular}}}}%
    \put(0.2278921,0.0055605){\color[rgb]{0,0,0}\makebox(0,0)[lt]{\lineheight{1.25}\smash{\begin{tabular}[t]{l}$e_j$\end{tabular}}}}%
    \put(0.48602873,0.00811631){\color[rgb]{0,0,0}\makebox(0,0)[lt]{\lineheight{1.25}\smash{\begin{tabular}[t]{l}$r_2$\end{tabular}}}}%
    \put(0.49028839,0.17935543){\color[rgb]{0,0,0}\makebox(0,0)[lt]{\lineheight{1.25}\smash{\begin{tabular}[t]{l}$g$\end{tabular}}}}%
    \put(0.71349544,0.00470857){\color[rgb]{0,0,0}\makebox(0,0)[lt]{\lineheight{1.25}\smash{\begin{tabular}[t]{l}$f$\end{tabular}}}}%
    \put(0.87962309,0.00896827){\color[rgb]{0,0,0}\makebox(0,0)[lt]{\lineheight{1.25}\smash{\begin{tabular}[t]{l}$r_3$\end{tabular}}}}%
  \end{picture}%
\endgroup%

   \end{center}
   \vspace{-3mm}
   \caption{Proof of Lemma \ref{ej}. Red edges are labelled by elements of $H_{i(j)}$.}\label{fig3}
 \end{figure}

\begin{lem}\label{Wgeod}
Suppose that $p$ is a path in $\G$ such that
$$
\lab(p)\equiv U_1 a_{1} U_2a_{2}\ldots U_{n}a_nU_{n+1},
$$
and the following conditions hold.
\begin{enumerate}
\item[(a)] For every $j=1, \ldots, n$,  $a_{j}$ is a letter in $H_{i(j)}$ for some $i(j)\in I$ and we have $\widehat\d_{H_{i(j)}}(1, a_{j}) > 5D$, where $D$ is the constant provided by Lemma \ref{Omega}.
\item[(b)] For every $j=1,\ldots , n+1$, $U_j$ is a (possibly empty) word in the alphabet $\cal A$ such that for any element $g\in G$ satisfying $H_{i(j-1)}gH_{i(j)}=H_{i(j-1)}U_jH_{i(j)}$, we have $\| U_j\| \le |g|_{\mathcal A}$. Here we assume $H_{i(0)}=H_{i(n+1)}=\{ 1\}$ for convenience.
\item[(c)] If $U_j$ is the empty word for some $j=2,\ldots , n$, then $H_{i(j-1)}\ne H_{i(j)}$.
\end{enumerate}
Then $p$ is geodesic.
\end{lem}

\begin{proof}
Arguing by contradiction, we fix the minimal $n\in \mathbb N$ such that there is a non-geodesic path $p$ satisfying all the assumptions. We fix the following notation for the segments of $p$:
$$
p=p_1e_1p_2e_2\ldots p_ne_np_{n+1},
$$
where
$$
\lab(p_j)\equiv U_j, \;\;\; {\rm and}\;\;\; \lab(e_j)\equiv a_j
$$
for all $j$ (if $U_j$ is the empty word, $p_j$ is the trivial path). We will first prove the following.

\begin{lem}\label{ej}
For every $j$, $e_j$ is an isolated $H_{i(j)}$-component of $p$.
\end{lem}

\begin{proof}
That each $e_j$ is an $H_{i(j)}$-component of $p$ follows immediately from (b) and (c). Let us prove that it is isolated. By the way of contradiction assume that $e_j$ is connected to another $H_{i(j)}$-component $f$ of $p$. For definiteness, let $p=r_1e_jr_2fr_3$ (see Fig. \ref{fig3}). Let also $g$ denote the edge of $\G $ labelled by an element of $H_{i(j)}$ and going from $(e_j)_+$ to $f_-$. Without loss of generality we can assume that the pair $e_j$, $f$ is chosen in such a way that $\ell (r_2)$ is minimal possible. If $r_2$ contains $e_{j+1}$, then $e_{j+1}$ cannot be isolated in the bigon $r_2g^{-1}$ by (a) and Lemma \ref{Omega} and we get a contradiction with the choice of $e_j$ and $f$. Therefore, $r_2$ is an initial subpath of $p_{j+1}$. Reading the label of the bigon $r_2g^{-1}$ we obtain
\begin{equation}\label{labp2}
\lab(r_2)\in H_{i(j)}.
\end{equation}
If $f$ is an edge of $p_{j+1}$, (\ref{labp2}) contradicts (b);  if $f=e_{j+1}$, then $U_2=r_2$ is empty by (\ref{labp2}) and (b), which contradicts (c).
\end{proof}

\begin{figure}
   \begin{center}
   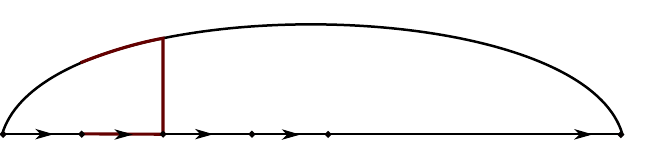
   \end{center}
   \vspace{-3mm}
   \caption{Proof of Lemma \ref{Wgeod}.}\label{fig3b}
 \end{figure}

We now return to the proof of Lemma \ref{Wgeod}. Let $q$ be a geodesic in $\G$ connecting $p_-$ to $p_+$. By our assumption, the subpath $p_2e_2\ldots p_ne_np_{n+1}$ of $p$ is geodesic. Thus we can think of $pq^{-1}$ as a geodesic quadrilateral with sides $p_1$, $e_1$, $p_2e_2\ldots p_ne_np_{n+1}$, and $q^{-1}$. By Lemma \ref{Omega} and condition (a), the $H_{i(1)}$-component $e_1$ cannot be isolated in $pq^{-1}$; by Lemma \ref{ej}, it is connected to a component $f_1$ of $q$ (see Fig. \ref{fig3b}).

Furthermore, let $q=q_1f_1 q^\prime$ and let $g$ be an edge of $\G$ connecting $(f_1)_+$ to $(e_1)_+$. Arguing as above, we conclude that $e_2$ cannot be isolated in the geodesic pentagon with sides $g$, $p_2$, $e_2$, $q_3e_3\ldots q_ne_np_{n+1}$, $(q^\prime)^{-1}$. By Lemma \ref{ej}, $e_2$ is connected to a component $f_2$ of $q^\prime$.  Thus, $q$ decomposes as $q_1f_1q_2f_2q^{\prime\prime}$. Continuing this process, we get a decomposition
$$
q=q_1f_1q_2f_2 \ldots q_nf_nq_{n+1},
$$
where each $f_j$ is an $H_{i(j)}$-component connected to $e_j$. Note that every $q_j$ has length at least $\ell (p_j)=\| U_j\|$ by (b) as $H_{i(j-1)}\lab(q_j) H_{i(j)}=H_{i(j-1)}U_j H_{i(j)}$. It follows that $\ell (q)\ge \ell(p)$, which contradicts the assumption that $p$ is not geodesic.
\end{proof}

\begin{prop}\label{scW}
Let $p$ be a path in $\G$. Suppose that $\lab(p)\in \mathcal H^\ast $ and the following conditions hold for all $i,j\in I$.
\begin{enumerate}
\item[(a)] If a letter $a\in H_i$ occurs in $\lab (p)$, then $\di(1, a) > 5D$, where $D$ is the constant provided by Lemma \ref{Omega}.
\item[(b)] If $a\in H_i$ and $b\in H_j$ are labels of two consecutive edges in $p$, then $H_{i}\cap H_{j}=\{1\}$ (in particular, $i\ne j$).
\end{enumerate}
Then $p$ is a $2$-attracting geodesic.
\end{prop}

\begin{proof}
It is easy to see that conditions (a) and (b) imply all the assumptions of Lemma \ref{Wgeod} (paths $p_j$ are trivial in this case). Thus, $p$ is geodesic.

To show that $p$ is $2$-attracting, let $x, y\in G$ be any vertices of $\G$ such that
\begin{equation}\label{Eq:daxyp}
\da(x, y)-\da(x,p)-\da(y,p)>2.
\end{equation}
Let $q$ be a geodesic going from $x$ to $y$. Let $u$, $v$ be vertices on $p$ such that
\begin{equation}\label{Eq:lslt}
\da(x,u)=\da(x,p)\;\;\; {\rm and}\;\;\; \da(y,v)=\da(y,p).
\end{equation}
Without loss of generality, we can assume that $p$ first passes through $u$ and then through $v$; we denote by $p_0$ the segment of $p$  between $u$ and $v$. Furthermore, we assume that $u$ and $v$ are chosen among all vertices satisfying (\ref{Eq:lslt}) so that $\ell(p_0)=\da(u,v)$ is minimal. We denote by $s$ and $t$ geodesics in $\G$ connecting $x$ to $u$ and $y$ to $v$, respectively.

\begin{figure}
   \begin{center}
   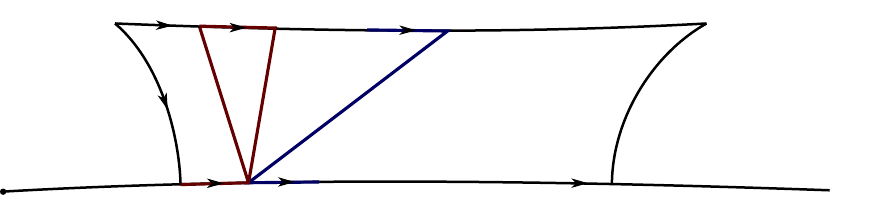
   \end{center}
   \vspace{-3mm}
   \caption{Proof of Proposition \ref{scW}. Edges of the same color are labeled by letters from the same $H_i$.}\label{fig4}
 \end{figure}

Note that every edge of $p_0$ is an $H_i$-component for some $i$, which cannot be isolated in the geodesic quadrilateral $sp_0t^{-1}q^{-1}$ by (a). Since $p$ is geodesic, every component of $p_0$ is isolated in $p_0$. Further, suppose that a component $e$ of $p_0$ is connected to a component $g$ of $s$. Then $\da(g_-, e_+)\le 1\le \da(g_-, u)$ and, consequently, $\da (x,u)=\da (x,e_+)$, which contradicts our assumption that the choice of $u$ and $v$ minimizes $\ell(p_0)$. Therefore, no component of $p_0$ is connected to a component of $s$. Similarly, we show that no component of $p_0$ is connected to a component of $t$.

Let
$$
p_0=e_1e_2\ldots e_n,
$$
where every $e_i$ is an edge. Note that
$$
\da(x,y)\le \da(x,u)+\da(u,v)+\da(v,y)=\da(x,p)+n+\da(y,p).
$$
Combining with (\ref{Eq:daxyp}) we obtain
\begin{equation}\label{Eq:nge2}
n\ge \da(x,y)-\da(x,p)-\da(y,p)\ge 2.
\end{equation}
As explained in the previous paragraph, $e_1$ is a component of $p_0$ connected to a certain component $f_1$ of $q$. Assume for definiteness that $e_1$, $f_1$ are $H_{i_1}$-components for some $i_1\in I$ and let $c_1$ denote the edge of $\G$ labelled by a letter from $H_{i_1}$ that connects $(f_1)_+$ to $(e_1)_+$.

Let $q=q^\prime f_1q_1$. The edge $e_2$ is a component of $p$, which cannot be isolated in the geodesic quadrilateral $c_1(e_2\ldots e_{n})t^{-1} (q_1)^{-1} $ by (a). Obviously, $e_2$ cannot be connected to $c_1$ by (b). As we showed above, $e_2$ cannot be connected to a component of $t$ either. Therefore, $e_2$ is connected to a component $f_2$ of $q_1$. Note that there exist edges in $\G$ connecting $(f_1)_-$ to $(e_1)_+=(e_2)_-$ and $(e_1)_+=(e_2)_-$ to $(f_2)_+$  (see Fig. \ref{fig4}). Therefore, we have $\da( (f_1)_-, (f_2)_+)=2$. Since $q$ is geodesic, $f_2$ must be the first edge of $q_1$. Assume for definiteness that $e_2$ and $f_2$ are $H_{i_2}$-components for some $i_2\in I$. Since $e_1$ is connected to $f_1$ and $e_2$ is connected to $f_2$, the vertices $a=(e_1)_+=(e_2)_-$ and $b=(f_1)_+=(f_2)_-$ belong to the same left $H_{i_1}$-coset and the same left $H_{i_2}$-coset of $G$. Therefore, $a^{-1}b \in H_{i_1}\cap H_{i_2}$. Using (b), we obtain $a=b$.

If $n>2$, we argue as above and show that $e_3$ is connected to the edge $f_3$ of $q$ next to $f_2$ and $f_2$ shares its endpoints with $e_2$. The latter condition means $e_2=f_2$. Continuing this way, we obtain a decomposition
$$
q=q^\prime f_1\ldots f_n q^{\prime\prime},
$$
where for each $2\le j\le n-1$, we have $f_j=e_j$. Thus, $p$ and $q$ share a common subpath of length at least $$n-2\ge \da(x,y)-\da(x,p)-\da(y,p)-2$$ (see (\ref{Eq:nge2})).
\end{proof}

\subsection{The $W(\xi, \sigma)$ small cancellation condition}\label{Sec:WSCC}

Constructing words satisfying the  $C(\mu, \e, \rho)$ and $C_1(\mu, \e, \rho)$ small cancellation conditions discussed in Section \ref{Sec:Ols} is not an easy task. Following  \cite{Osi10}, we now introduce an easily verifiable combinatorial condition that implies $C_1(\mu, \e, \rho)$.

Throughout this section, we fix a group $G$, a hyperbolically embedded collection of subgroups $\Hl$, and some $X\subseteq G$ such that $\Hl\h (G,X)$. Note that the latter assumption is stronger that condition ($\ast$) from the previous section. Furthermore, we let $\mathcal H$ and $\mathcal A$ denote the alphabets defined by (\ref{calA}).  We say that two letters $a$, $b$ of a word $W\in \mathcal A^\ast$ are \emph{cyclically consecutive} if they are consecutive or if $a$ (respectively, $b$) is the last (respectively, first) letter of $W$.

\begin{defn}\label{Defn:WMlr}
A set of words $\mathcal R\subseteq \mathcal H^\ast$ satisfies the \emph{$W(\xi, \sigma)$ condition} for some $\xi, \sigma \ge 0$ if the following hold.

\begin{enumerate}
\item[(W$_1$)] If $a\in H_i$ and $b\in H_j$ are cyclically consecutive letters of some word from $\mathcal R$, then $H_{i}\cap H_{j}=\{1\}$.
\item[(W$_2$)] If a letter $a\in H_i$ occurs in some word from $\mathcal R$, then $\di(1, a) \ge \xi $.
\item[(W$_3$)] For each letter $a\in \mathcal H$, there is at most one occurrence of $a^{\pm 1}$ in all words from $\mathcal R$. More precisely, let $R, S\in \mathcal R$. Suppose that $R\equiv R_1aR_2$ and $S\equiv S_1a^\e S_2$ for some $a\in \mathcal H$, $R_1, R_2, S_1, S_2\in \mathcal A$, and $\e=\pm 1$. Then $\e=1$, and $R_i\equiv S_i$ for $i=1,2$; in particular, $R\equiv S$.
\item[(W$_4$)] For every $R\in \mathcal R$, we have $\| R\| \ge \sigma$.
\end{enumerate}
\end{defn}

\begin{rem}\label{Rem:KG}
Although the condition $W(\xi,\sigma)$ makes sense for all groups with hyperbolically embedded subgroups, it is not really useful for groups with non-trivial finite radical. Indeed, in non-degenerate cases (e.g., if all $H_i$ are infinite), one can use Corollary \ref{Cor:ICC}, (W$_1$), and (W$_4$) to derive that the existence on a non-empty set $\mathcal R$ satisfying $W(\xi, \sigma)$ for some $\xi\ge 1$ and $\sigma \ge 2$ implies $K(G)=\{ 1\}$.
\end{rem}

We consider a particular example. Suppose that $\Hl=\{ A, B, ...\} $, $A\cap B=\{1\}$, and there exist infinite order elements $a\in A$ and $b\in B$. For any $m,n\in \mathbb N$, we think of $a^m$ and $b^n$ as single letters from the alphabet $\mathcal H$ that belong to $A$ and $B$, respectively. For any $\ell \in \mathbb N$, let $\mathcal P(\ell)$ denote the set of all words
\begin{equation}\label{Eq:Si}
P_i = a^{i\ell+1}b^{i\ell+1} a^{i\ell+2}b^{i\ell+2} \ldots a^{i\ell+\ell}b^{i\ell+\ell}, \;\;\;\;\; i\in \NN.
\end{equation}

\begin{lem}\label{Lem:sc}
For every $\xi, \sigma >0$, there exists $N\in \NN$ such that for any integer $\ell \ge N$, the set $\mathcal P(\ell)$ satisfies the $W(\xi, \sigma )$ condition.
\end{lem}

\begin{proof}
Conditions (W$_1$) and (W$_3$) obviously hold for each $P_i$. Since $\Hl\h G$ and the orders of $a$ and $b$ are infinite, we have $$\lim_{n\to \infty} \widehat{\d}_{A}(1, a^n)=\lim_{n\to \infty} \widehat{\d}_{B}(1, b^n)=\infty$$ by part (b) of Definition \ref{hedefn}. Therefore, we can ensure that (W$_2$) holds for any given $\xi>0$ by taking $\ell $ large enough. Clearly, $\| P_i\|=2\ell$. Thus we can also ensure (W$_4$) for any given $\sigma$ by taking $\ell\ge \sigma/2$.
\end{proof}

The following theorem relates $W(\xi,\sigma)$ to the small cancellation conditions defined in Definition \ref{DefSC}. Recall that a set of words in some alphabet is said to be \emph{symmetric} if it is closed under taking inverses and cyclic shifts. The \emph{symmetrization} of a set of words $\mathcal R$ is the smallest symmetric set of words containing $\mathcal R$; clearly, it coincides with the set of all cyclic shifts of $R\in \mathcal R$ and their inverses.

\begin{lem}\label{W->C}
For any positive constants $\e$, $\mu$, and $\rho$, there exist positive $\xi$ and $\sigma$ such that, for any set of words $\mathcal W=\{ W_j\}_{j\in J}\subseteq \mathcal A^\ast$ satisfying $W(\xi, \sigma)$, the following hold.

\begin{enumerate}
\item [(a)] The symmetrization of $\mathcal W$ satisfies  $C_1(\e,\mu,\rho)$.
\item [(b)] For every $j\in J$, suppose that the first (respectively, last) letter of $W_j$ belongs to $H_{j_1}$ (respectively, $H_{j_2}$) for some $j_1, j_2\in I$. Let $\{ x_j\}_{j\in J}$ be a set of letters from $X$ such that the element represented by $x_j$ in $G$ does not belong to $H_{j_2}H_{j_1}$. Then the symmetrization of the set $\{ x_jW_j\}_{j\in J}$ satisfies  $C_1(\e,\mu,\rho)$.
\end{enumerate}
\end{lem}

\begin{rem}
Recall that $C_1(\e,\mu,\rho)$ implies $C(\e,\mu,\rho)$. Thus, the sets described in (a) and (b) also satisfy $C(\e,\mu,\rho)$.
\end{rem}

\begin{proof}
We first prove (a). Let $\e$, $\mu$, and $\rho$ be any positive numbers. Let $\xi$ be any constant satisfying
\begin{equation}\label{x5D}
\xi > 5D,
\end{equation}
where $D$ is the constant from Lemma \ref{Omega}. Further, we choose large enough $\sigma$ so that
\begin{equation}\label{Eq:sigma}
\sigma > \max\left\{ \frac{2\e+2}{\mu}, \rho\right\}.
\end{equation}

Let $\mathcal W$ be as above and let $\mathcal R$ denote the symmetrization of $\mathcal W$. Conditions (W$_1$), (W$_2$), and inequality (\ref{x5D}) ensure that every path in $\G$ labeled by a word $R\in \mathcal R$ satisfies the assumptions of Proposition \ref{scW}. Therefore, every such $R$ is $(G, \mathcal A)$-geodesic. Combining this with the inequality $\sigma>\rho$ and (W$_4$), we conclude that condition (a) from Definition \ref{DefSC} holds.

 \begin{figure}
   \begin{center}
   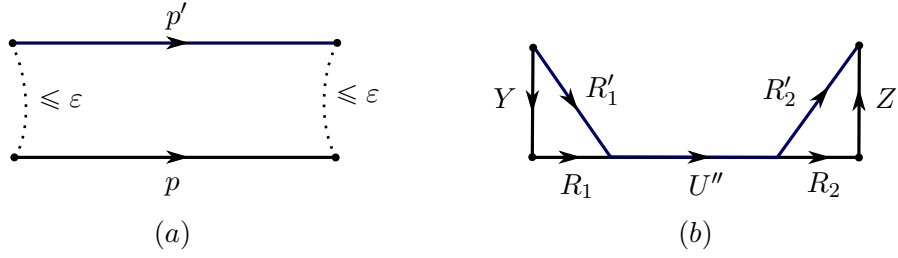
   \end{center}
   \vspace*{-3mm}
   \caption{The proof of the implication $W(\xi, \sigma)\Rightarrow C_1(\e,\mu, \rho)$.}\label{fig5}
 \end{figure}

Further, suppose that there are two relations $R,R^\prime \in \mathcal R$ such that $R\equiv UV$, $R^\prime \equiv U^\prime V^\prime$, $U^\prime = YUZ$ in $G$ for some words $Y,Z\in \mathcal A^\ast$, and inequalities (\ref{piece1}), (\ref{piece2}) hold. Without loss of generality, we can assume that $\| U^\prime \| \ge \| U\|$; therefore, (\ref{piece1}) can be rewritten as
\begin{equation}\label{U'RR'}
\| U^\prime\| \ge \mu \min\{ \| R\|, \| R^\prime\| \}.
\end{equation}
Translating these assumptions to geometric language, we can find paths $p$ and $p^\prime$ in $\G$ such that
$$
\lab(p)\equiv U,\;\;\; \lab (p^\prime )\equiv U^\prime,
$$
and
$$
\max\{\da (p_-, p^\prime _-), \da (p_+, p^\prime_+)\} \le \e
$$
(see Fig. \ref{fig5} (a)). By (W$_1$), (W$_2$), and Proposition \ref{scW}, $p$ is a $2$-attracting geodesics. Therefore, $p^\prime$ and $p$ share a common subpath $p^{\prime\prime}$ of length at least
\begin{equation}\label{Eq:lqpp}
\begin{array}{rcl}
\ell(p^{\prime\prime}) & \ge & \ell(p^\prime) - \da (p^\prime_-, p)- \da (p^\prime_+, p)-2 \ge \\&&\\ &&
\|U^\prime\| - 2\e -2 \ge \mu \min\{ \| R\|, \| R^\prime\| \} - 2\e -2\ge \mu \sigma -2\e-2 >0.
\end{array}
\end{equation}
(We use (\ref{U'RR'}) and (\ref{Eq:sigma}) here.)

Let $\lab(p^{\prime\prime})\equiv U^{\prime\prime}$, $\lab (p)\equiv R_1U^{\prime\prime}R_2$, $\lab (p^\prime)\equiv R_1^\prime U^{\prime\prime}R_2^\prime $ (see Fig. \ref{fig5} (b)). Condition (W$_3$) implies that the cyclic shifts of $R$ and $R^\prime$ starting from $U^{\prime\prime}$ coincide; that is, we have $U^{\prime\prime}R_2VR_1\equiv U^{\prime\prime}R_2^\prime V^\prime R_1^\prime$. Reading the label of the leftmost triangle on Fig. \ref{fig5} (b), we obtain $YR_1=R_1^\prime$ in $G$. Therefore,
$$
YR = YR_1U^{\prime\prime}R_2V= R_1^\prime U^{\prime\prime}R_2VR_1R_1^{-1} =R_1^\prime U^{\prime\prime}R_2^\prime V^\prime R_1^\prime R_1^{-1}=R^\prime Y
$$
in $G$. Thus, $YRY^{-1}=R^\prime$ in $G$ and $\mathcal R$ satisfies condition (b) in Definition \ref{DefSC}.

Finally, suppose that some word $R\in \mathcal R$ contains two disjoint subwords $U$ and $U^\prime$ such that $U^\prime = YUZ$ or $U^\prime = YU^{-1}Z$ in $G$ for some words $Y$, $Z$ and (\ref{piece1}), (\ref{piece2}) hold. Arguing as above, we can find two disjoint occurrences of the same word $U^{\prime\prime}$ in $R$, which contradicts (W$_3$). Thus, condition (c) in Definition \ref{DefSC} holds.

To prove (b), we increase the parameter $\sigma$ so that, in addition to (\ref{Eq:sigma}), it satisfies
\begin{equation}\label{Eq:sigma1}
\sigma\mu > 1+\max\{ 10^3(2\e+2), 8c\},
\end{equation}
where $\delta $ is the hyperbolicity constant of $\G$ and
$$
c =  \max\{ 10^3(26\delta+2), 3 \cdot 10^4 \delta\}.
$$

Let $\mathcal S$ denote the symmetrization of the set $\{ x_jW_j\}_{j\in J}$. We first observe that every $S\in \mathcal S$ satisfies the assumptions of Lemma \ref{Wgeod} (note that $x_j\notin H_{j_2}H_{j_1}$ guarantees condition (b) of the lemma). Therefore, every $S\in \mathcal S$ is $(G, \mathcal A)$-geodesic.

Further, suppose that there are two relations $S,S^\prime \in \mathcal S$ such that $$S\equiv UV,\;\;\; S^\prime \equiv U^\prime V^\prime,$$ $U^\prime = YUZ$ in $G$ for some words $Y,Z\in \mathcal A^\ast$, and inequalities (\ref{piece1}), (\ref{piece2}) hold. As above, we assume that $\| U^\prime \| \ge \| U\|$ and rewrite (\ref{piece1}) as in (\ref{U'RR'}).

The difference with part (a) is that the words $U$ and $U^\prime$ may contain letters from the set $\{ x_j\}_{i\in I}$ and thus the corresponding paths $p$ and $p^\prime$ are not necessarily $2$-attracting. However, we can overcome this difficulty as follows. Let $aqb(q^\prime)^{-1}$ be a geodesic quadrilateral in $\G$ such that
$$
\lab(q)\equiv U,\;\;\; \lab (q^\prime )\equiv U^\prime,\;\;\; \lab(a)\equiv Y, \;\;\; \lab(b)\equiv Z.
$$
By (\ref{piece2}), we have
$$
\max\{\ell(a), \ell(b)\} \le \e
$$
Note that $q$ and $q^\prime$ are geodesic. We subdivide them as follows. If $U$ contains some $x\in \{ x_j^{\pm 1}\}_{i\in I}$, i.e., $U\equiv U_1x U_2$, we let $q=q_1rq_2$, where $\lab (q_1)\equiv U_1$, $\lab (r) \equiv x$, and $\lab (q_2)\equiv U_2$; otherwise, we let $q_1=q$ and let $r$, $q_2$ be the trivial paths. Similarly, we define a subdivision $q^\prime=q_1^\prime r^\prime q_2^\prime$ depending on whether $U^\prime $ contains a letter from $\{ x_j^{\pm 1}\}_{i\in I}$ or not (see Fig. \ref{Subdiv}).

Further, we divide the sides of the geodesic octagon $aq_1rq_2b(q_1^\prime r^\prime q_2^\prime)^{-1}$ (some sides may be trivial) into the sets
$$
A=\{ q_1, q_2, q_1^\prime, q_2^\prime\},\;\;\; {\rm and }\;\;\; B=\{ a,b, r, r^\prime\}
$$
and let $\alpha$ (respectively, $\beta$) denote the sum of lengths of sides from the set $A$ (respectively, $B$). By our construction, we have $\beta \le 2\e +2$. By \eqref{U'RR'}, (W$_4$), and \eqref{Eq:sigma1}, we have
$$
\alpha \ge \ell (q^\prime)- 1=\| U^\prime\| -1 \ge \mu \min\{ \| R\|, \| R^\prime\| \} -1\ge  \mu\sigma -1 > \max\{ 10^3\beta, 8c\}.
$$
Thus, we can apply Lemma \ref{N123} and find two segments $p$ and $p^\prime$ of two distinct sides from $A$ such that $$\ell (p^\prime)> 10^{-3}c\ge 26\delta +2$$ and $$\max\{\d(p_-, p^\prime_-),\d(p_+, p^\prime_+)\}\le 13\delta.$$ Since a letter from $\{ x_j^{\pm 1}\}_{i\in I}$ can occur in $\lab(q)$ at most once, the word $\lab(p)$ does not contain any such letters. As in part (a), we obtain that $p$ is a $2$-attracting geodesic by Proposition \ref{scW}. Therefore, $p$ and $p^\prime$ share a common subpath $p^{\prime\prime}$ of length at least
$$
\ell(p^{\prime\prime}) \ge \ell(p^\prime) - \da (p^\prime_-, p)- \da (p^\prime_+, p)-2 > 26\delta +2-13\delta - 13\delta -2 =0.
$$
Since $q$ and $q^\prime$ are geodesic,  $p$ and $p^\prime$ cannot simultaneously be subpaths of $q$ or $q^\prime$. Without loss of generality, we can assume that $p$ is a subpath of $q$ and $p^\prime$ is a subpath of $q^\prime$. Now, the equality $YRY^{-1}\equiv R^\prime$ in $G$ is derived exactly as in part (a). This gives us condition (b) in Definition \ref{DefSC}.

\begin{figure}
   \begin{center}
   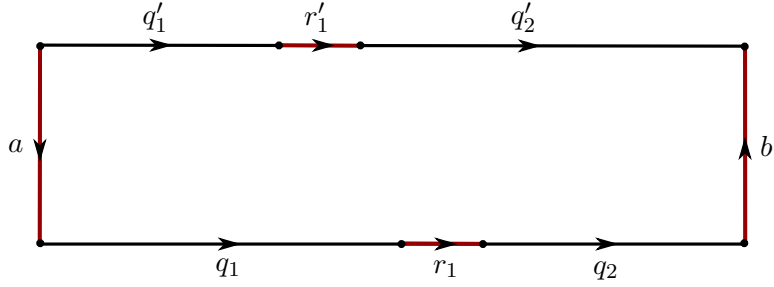
   \end{center}
   \vspace*{-3mm}
   \caption{The geodesic octagon $aq_1rq_2b(q_1^\prime r^\prime q_2^\prime)^{-1}$. Sides from the set $B$ are highlighted in red.}\label{Subdiv}
 \end{figure}

The proof of condition (c) in Definition \ref{DefSC} is similar. Assuming that some word $S\in \mathcal S$ contains two disjoint subwords $U$ and $V$ such that $U^\prime = YUZ$ or $U^\prime = YU^{-1}Z$ in $G$ and (\ref{piece1}), (\ref{piece2}) hold, we argue as above and find two disjoint occurrences of a certain non-empty word in $R$. This contradicts (W$_3$).
\end{proof}

We conclude this section with the analogue of the Greendlinger lemma for the $W(\xi,\rho)$ small cancellation condition.

\begin{prop}\label{Prop:GL}
Let $G$ be a group, $\Hl$ a collection of subgroups of $G$, $X$ a subset of $G$ such that $\Hl\h (G,X)$. Let $\mathcal H$ and $\mathcal A$ denote the alphabets defined by (\ref{calA}). For every $\nu >0$, there exist constants $\xi, \sigma > 0$ such that the following holds.

Let $\mathcal R$ be a set of words in the alphabet $\mathcal H$ satisfying $W(\xi, \sigma)$. Suppose that a non-trivial $(G,\mathcal A)$-geodesic word $W\in \mathcal A^\ast$ represents the identity in $G/\ll \mathcal R\rr $. Then there exists $R\in \mathcal R$ such that a cyclic shift of $R^{\pm 1}$ and $W$ have a common subword $U$ of length $\| U\| \ge (1-\nu)\| R\|$.
\end{prop}

\begin{proof}
We choose any $\mu \in (0, 1/16]$ such that $\mu < \nu/13$.
Further, let $\rho$, $\e$ be the corresponding constants provided by Lemma \ref{GOL} for the group $G$. By Lemma \ref{W->C}, we can choose  $\xi$ and $\sigma$ such that, for any set of words $\mathcal R$ in the alphabet $\mathcal H$ satisfying $W(\xi, \sigma)$, the symmetrization of $\mathcal R$ satisfies $C(\e, \mu, \rho)$. Note that increasing $\xi$ and $\sigma$ only makes the condition $W(\xi, \sigma)$ stronger. Thus we can assume that $\xi>5D$, where $D$ is provided by Lemma \ref{Omega},  $\sigma \ge \rho $, and
\begin{equation}\label{Eq:mu1}
\mu < \nu/13 -\frac{4\e +2}{13\sigma}.
\end{equation}
We fix any presentation of $G$ of the form \ref{GA0}.

By Lemma \ref{GOL}, for every $(G,\mathcal A)$-geodesic word $W\in \mathcal A^\ast$ representing the identity in $G/\ll \mathcal R\rr $, there exists a van Kampen diagram $\Delta$ over the presentation (\ref{quot}), an $\mathcal R$-face $\Pi$ in $\Delta$, and an $\e$-contiguity diagram $\Gamma $ of $\Pi $ to  $\partial \Delta $ such that $(\Pi, \Gamma, \partial \Delta)\ge 1-13\mu$. We keep the notation introduced on Fig. \ref{fig1} for the boundary sides of $\Gamma$. In this notation, we have
\begin{equation}\label{Eq:lq1}
\begin{array}{rcl}
\ell(q_1)& = &(\Pi, \Gamma, \partial \Delta)\ell (\partial \Pi)>(1-13\mu) \ell (\partial \Pi)\ge (1- \nu +\frac{4\e+2}\sigma) \ell (\partial \Pi)\ge\\&&\\&& (1-\nu)\ell (\partial \Pi) +4\e+2
\end{array}
\end{equation}
(to derive the last two inequalities, we use (\ref{Eq:mu1}) and (W$_4$)).

Recall that $\Gamma $ is a diagram over (\ref{GA0}). Therefore, $\partial \Gamma$ can be mapped to $\G$ by a combinatorial map preserving labels and orientation of edges. We keep the notation $s_1q_1s_2q_2$ for the image of $\partial \Gamma$ in $\G$. Since $q_2$ is labelled by a subword of $W$, it is geodesic in $\G$. By Proposition \ref{scW}, $q_1$ is a $2$-attracting geodesic in $\G$. Therefore, $q_1$ and $q_2^{-1}$ share a subpath $q$ of length al least
$$
\ell(q) \ge \ell(q_2) -\ell(s_1)-\ell(s_2) -2 \ge \ell (q_1) - 2\ell(s_1)-2\ell(s_2) -2 \ge \ell (q_1) -4\e -2.
$$
Combining this inequality with (\ref{Eq:lq1}), we obtain that $W$ and $\lab(\partial \Pi)$ share a common subword $U\equiv\lab(q)$ of length at least $(1-\nu)\| \lab(\partial \Pi)\|$.
\end{proof}

\subsection{Suitable subgroups and quotients}

In this section, we summarize results from \cite{MO,Hull,Osi10} and obtain some new results necessary to prove Theorems \ref{app1} and \ref{app2}. We begin by recalling the notion of a suitable subgroup. It was originally introduced in \cite{Osi10} for relatively hyperbolic groups and then generalized to groups with hyperbolically embedded subgroups in \cite{Hull}. In our paper, we will use this notion in both settings.

We begin with a definition from \cite[Section 5]{Hull}.

\begin{defn}\label{Def:suit1}
Let $G$ be a group, $\mathcal A$ a generating alphabet such that the Cayley graph $\G$ is hyperbolic.
A subgroup $S\le G$ is said to be {\it $\mathcal A$-suitable} (suitable with respect to $\mathcal A$ in the terminology of \cite{Hull}) if the action of $S$ on $\G$ is non-elementary and $S$ does not normalize any non-trivial finite subgroup of $G$.
\end{defn}

Recall that two elements $a$, $b$ of a group are commensurable if $a^m$ is conjugate to $b^n$ for some non-zero integers $m$ and $n$.

\begin{lem}[{\cite[Corollary 5.7]{Hull}}]\label{non-comm}
Let $G$ be a group, $\mathcal A$ a generating alphabet of $G$ such that $\G$ is hyperbolic and the action of $G$ on $\G$ is acylindrical. Suppose that $S$ is an $\mathcal A$-suitable subgroup of $G$. Then, for every $n\in \NN$, there exist pairwise non-commensurable elements $g_1, \ldots, g_n\in S$  such that $g_i$ acts loxodromically on $\G$ and $E(g_i)=\langle g_i\rangle $ for all $i=1, \ldots, n$.
\end{lem}

For relatively hyperbolic groups, there is a similar notion of suitability defined as follows \cite{AMO,Osi10}.

\begin{defn}\label{Def:suit2}
Let $G$ be a group hyperbolic relative to a collection of subgroups $\Hl$. A subgroup $S\le G$ is \emph{suitable with respect to the peripheral collection} $\Hl$ if it is not virtually cyclic, contains a loxodromic element, and does not normalize any non-trivial finite subgroup of $G$.
\end{defn}

Recall that, in the settings of Definition \ref{Def:suit2}, an element $g\in G$ is loxodromic with respect to the action on $\G$ if and only if $|g|=\infty $ and $g$ is not conjugate to an element of one of the peripheral subgroups. The definition of a suitable subgroup of a relatively hyperbolic group was first formulated in \cite{Osi10} in a slightly different way. Later it was shown to be equivalent to Definition \ref{Def:suit2} in \cite{AMO} (see Lemma 3.3 and Proposition 3.4 there). Furthermore, Definitions \ref{Def:suit1} and \ref{Def:suit2} are equivalent in the following sense.

\begin{lem}\label{Lem:seq}
Let $G$ be a group, hyperbolic relative to a finite collection of subgroups $\Hl$, $X$ a finite relative generating set of $G$ with respect to $\Hl$, $\mathcal A$ the alphabet defined by (\ref{calA}). A subgroup $S\le G$ is suitable with respect to the peripheral collection $\Hl$ if and only if $S$ is $\mathcal A$-suitable.
\end{lem}

\begin{proof}
By Theorem \ref{Thm:AH}, the action of $G$ on $\G$ is acylindrical. If $S$ is $\mathcal A$-suitable, it contains a loxodromic element by Lemma \ref{non-comm} and is not virtually cyclic because its action on $\G $ is non-elementary (see Theorem \ref{Thm:class}). Thus, $S$ is suitable with respect to the peripheral collection $\Hl$. Conversely, if $S$ is suitable with respect to the peripheral collection $\Hl$, then its action on $\G$ must be non-elementary by Theorem \ref{Thm:class}; hence, $S$ is $\mathcal A$-suitable.
\end{proof}

The next result provides a sufficient condition for a subgroup to be suitable.

\begin{lem}\label{Lem:suit1}
Let $G$ be a group and $\mathcal A$ a generating alphabet of $G$ such that $\G$ is hyperbolic and the action of $G$ on $\G$ is acylindrical and non-elementary. If $K(G)=\{ 1\}$, then every non-trivial normal subgroup of $G$ is $\mathcal A$-suitable.
\end{lem}

\begin{proof}
Let $S$ be a non-trivial normal subgroup of $G$. Since $G$ contains no non-trivial finite normal subgroups, $S$ must be infinite. By Lemma \ref{Lem:NormAH}, the action of $S$ on $\G$ is non-elementary. Acylindricity of the action of $G$ on the Cayley graph $\G$ allows us to apply \cite[Lemma 5.5]{Hull}, which claims the existence of a (unique) maximal finite subgroup $E\le G$ normalized by $S$. Since $S$ is normal in $G$, so is $E$. In particular, $E\le K(G)$, which implies $E=\{ 1\}$. Thus, $S$ is $\mathcal A$-suitable.
\end{proof}

\begin{lem}\label{Lem:suit2}
Let $G$ be a group, $\Hl$ a collection of subgroups of $G$, $X$ a subset of $G$ such that $\Hl\h (G,X)$, $\mathcal A$ an alphabet defined by (\ref{calA}), $S$ a subgroup of $G$. Suppose that the action of $G$ on $\G$ is acylindrical and there exist $i_0,j_0\in I$ such that $H_{i_0}\cap H_{j_0}=\{ 1\}$ and the intersections $S\cap H_{i_0}$, $S\cap H_{j_0}$ are infinite. Then $S$ is $\mathcal A$-suitable.
\end{lem}

\begin{proof}
Let $D$ be the constant provided by Lemma \ref{Omega} applied to $\Hl\h (G,Y)$ and let $\widehat\d_{H_i}$ be the corresponding generalized metrics on groups $H_i$. Since the intersections $S\cap H_{i_0}$ and $S\cap H_{j_0}$ are infinite, there exist $a\in S\cap H_{i_0}$, $b\in S\cap H_{j_0}$ such that $\widehat \d _{H_{i_0}}(1, a)>5D$ and $\widehat \d _{H_{j_0}}(1, b)>5D$. By Lemma \ref{Wgeod}, every path in $\G$ labeled by $(ab)^n$ for some $n\in \NN$ is geodesic. Therefore, the action of $S$ on $\G$ has unbounded orbits. Note also that $S$ cannot be virtually cyclic; for otherwise, any two infinite subgroups of $S$ would have non-trivial intersection, which contradicts the assumption $H_{i_0}\cap H_{j_0}=\{ 1\}$. Therefore, the action of $S$ on $\Gamma (G, \mathcal A)$ is non-elementary by Theorem \ref{Thm:class}. Finally, we note that $S$ does not normalize any non-trivial finite subgroup of $G$ by Corollary \ref{Cor:ICC}.
\end{proof}

To make our paper as self-contained as possible, we summarize some results from \cite{Hull} and \cite{MO} used below. The first one is proved in \cite[Proposition 3.3]{MO}. We simplify it and change the notation a bit to better fit the other results discussed below.

\begin{lem} \label{SCQ} Let $G$ be a  group and let $\Hl$, $\{ A,B\}$ be two collections of subgroups of $G$. Suppose that $A\cap B=\{ 1\}$ and $\Hl\cup \{A,B\} \h(G,X)$ for some $X\subseteq G$. Then there exists $n \in \mathbb N$ and finite subsets $\mathcal F_A\subseteq A$, $\mathcal F_B\subseteq B$  such that the following holds.

Suppose that $\mathcal W=\{ W_m\}_{m\in \NN}$ is an arbitrary set of words in the alphabet $X\sqcup A \sqcup B$ of the form
\begin{equation}\label{W}
W_m\equiv x_m a_{m1} b_{m1}\ldots a_{mn}b_{mn}
\end{equation}
satisfying the following conditions for every $m \in \NN$:
\begin{enumerate}
\item[(a)] $x_m\in X$;
\item[(b)] $a_{m\ell}\in A\setminus \mathcal F_A$, $b_{m\ell}\in B\setminus \mathcal F_B$ for all $1\le \ell\le n$;
\item[(c)] if a letter $c\in A\sqcup B$ occurs in $W_m$ for some $m\in \NN$, then it occurs only in $W_m$ and only once; in addition, $c^{-1}$ does not occur in any word from $\mathcal W$.
\end{enumerate}
Then the restriction of $\gamma\colon G\rightarrow G/\ll \mathcal W\rr$ to each $H_i$ is injective and $\{ \gamma(H_i)\}_{i\in I}\h G/\ll \mathcal W\rr$.
\end{lem}

The next result combines Lemmas 4.4 and Lemma 4.9 from \cite{Hull} (note that the result about relative hyperbolicity in part (b) follows form $\{\gamma(H_i)\}_{i\in I}\h (\overline G, \gamma(X))$ and Proposition \ref{rhhe}; it was first proved in  \cite[Lemma 5.1]{Osi10}). As everywhere in this paper, we reproduce the simplified versions of the above-mentioned results for the $C(\e, \mu, \rho)$ small cancellation condition instead of the more general condition $C(\e, \mu, \lambda, c, \rho)$ considered in \cite{Hull,Osi10}. Note that, unlike in the previous proposition, the set $\mathcal R$ is assumed to be finite here.

\begin{lem}\label{Lem:Hull}
Let $G$ be a group, $\Hl$ a collection of subgroups of $G$ such that $\Hl\h (G,X)$ for some $X\subseteq G$, $\mathcal A$ an alphabet defined by (\ref{calA}). There exist $\e\ge 0$ and $\mu, \rho>0$ such that, for any finite symmetric set of words $\mathcal R$ in the alphabet $\mathcal A$ satisfying $C(\e, \mu , \rho )$, the following hold.
\begin{enumerate}
\item[(a)] The restriction of the natural homomorphism $\gamma\colon G\to \overline G=G/\ll \mathcal R\rr$ to the set
\begin{equation}\label{BN}
B_N= \{ g\in G\mid |g|_\mathcal A\le N\}
\end{equation}
is injective. In particular, the restriction of $\gamma$ to $\bigcup\limits_{i\in I} H_i$ is injective.
\item[(b)] $\{\gamma(H_i)\}_{i\in I}\h (\overline G, \gamma(X))$. In particular, if $G$ is hyperbolic relative to $\Hl$ and $|X|<\infty$, then $\overline{G}$ is hyperbolic relative to $\{\gamma(H_i)\}_{i\in I}$.
\item[(c)] Every finite order element of $\overline{G}$ is the image of a finite order element of $G$.
\end{enumerate}
\end{lem}
It is worth noting that in the statement of \cite[Lemma 4.4]{Hull}, Hull only claims that $\{\gamma(H_i)\}_{i\in I}\h \overline G$ in the settings of Lemma \ref{Lem:Hull}. However, the proof (which is essentially the same as the one of \cite[Lemma 5.1]{Osi10}) actually shows that $\{\gamma(H_i)\}_{i\in I}\h (\overline G, \gamma(X))$.

We are now ready to state the main result of this section. It generalizes and strengthens analogous results of \cite{Hull,Ols93,Osi10}. Compared to \cite{Hull}, the main improvement is that we allow the set $\mathcal F$ to be infinite. This generalization is essential for some applications, e.g., for the proof of Theorem \ref{app2}. The result about centralizers (part (d)) is also new in these settings. We provide a complete proof modulo results of \cite{Hull,MO} discussed above to make sure that these new additions are consistent with previously known parts.

\begin{thm}\label{glue}
Let $G$ be a group, $\Hl$ a collection of subgroups of $G$ such that $\Hl\h (G,X)$ for some $X\subseteq G$, $\mathcal A$ an alphabet defined by (\ref{calA}). Suppose that the action of $G$ on $\G$ is acylindrical. For any $\mathcal A$-suitable subgroup $S\le G$, any countable subset $\mathcal F \subseteq X$, and any $N\in \NN$, there exists a subset $\{ s_f\mid f\in \mathcal F\}\subseteq S$ such that the quotient group
\begin{equation}\label{Eq:Gbar}
\overline{G}=G/\ll \mathcal R \rr,\;\; {\rm where }\; \mathcal R= \{fs_f \mid f\in \mathcal F\},
\end{equation}
satisfies the following conditions.
\begin{enumerate}
\item[(a)]  The restriction of the natural homomorphism $\gamma\colon G\to \overline{G}$ to the set $B_N$ defined by (\ref{BN}) is injective. In particular, $\gamma$ is injective on $\bigcup \limits_{i\in I} H_i$.
\item[(b)] $\{\gamma(H_i)\}_{i\in I}\h \overline G$.
\item[(c)] Every finite order element of $\overline{G}$ is the image of a finite order element of $G$.
\item[(d)] For every $g\in B_N$, we have $C_{\overline{G}}(\gamma(g))=\gamma(C_G(g))$.
\item[(e)] If, in addition, the sets $I$, $\mathcal F$, and $X$ are finite, then $\overline G$ is hyperbolic relative to $\{\gamma(H_i)\}_{i\in I}$.
\end{enumerate}
\end{thm}

\begin{rem}
Note that, in the notation of Theorem \ref{glue}, we have $\gamma(\mathcal F)\subseteq \gamma(S)$ since $s_f\in S$ for all $f\in \mathcal F$.
\end{rem}

\begin{proof}
By Lemma \ref{non-comm}, there exist non-commensurable elements $a, b\in S$ acting loxodromically on $\G$ such that $E(a)=\langle a\rangle$, $E(b)=\langle b\rangle$. Note that $E(a)\cap E(b) =\{ 1\}$ by Proposition \ref{Prop:maln}. Furthermore, we have $\{ \langle a\rangle, \langle b\rangle\} \cup \Hl \h (G,X)$ by Theorem \ref{Thm:Eg}. We use the notation $A=\langle a\rangle $, and $B=\langle b\rangle$, which is consistent with Lemma \ref{Lem:Hull}.

Let $\mathcal F=\{ f_1, f_2, \ldots\}$. Since $C_1(\e, \mu, \rho)$ implies $C(\e, \mu, \rho)$ and both conditions become stronger as $\e$ and $\rho $ increase and $\mu$ decreases, we can choose sufficiently large $\e$, $\rho$ and sufficiently large $\mu$ such that the conclusions of Lemmas \ref{Prop:HO} and \ref{Lem:Hull} simultaneously hold. Further, let $\xi$ and $\sigma $ be the constants provided by Lemma \ref{W->C} for the chosen values of $\e$, $\mu$, and $\rho$.
By Lemma \ref{Lem:sc}, There exists $\ell\in \NN$ such that the set of words $\mathcal P(\ell)=\{ P_i\mid i\in \NN\}$ defined by (\ref{Eq:Si}) satisfies $W(\xi,\sigma)$. Combining this with  Lemma \ref{W->C} (b), we conclude that the symmetrization of the set $\{ f_1P_1, f_2P_2, \ldots\}$ satisfies the $C_1(\e, \mu, \rho)$ small cancellation condition. In addition, increasing $\ell$ if necessary, we can ensure that all letters in the words $P_i$ come from the sets $A\setminus \mathcal F_A$ and $B\setminus \mathcal F_B$, where $\mathcal F_A$ and $\mathcal F_A$ are provided by Lemma  \ref{SCQ} (we use the condition $\{ A, B\}\h G$ and part (b) of Definition \ref{hedefn} here).

Since $a,b\in S$, each $P_i$ represents and element of $S$ in $G$. Thus, it suffices to show that claims (a)--(e) hold for the group
$$
\overline{G}=G/\ll f_1P_1, f_2P_2, \ldots \rr.
$$
Thanks to the choice of the parameters explained in the previous paragraph, we can apply Lemmas~\ref{Prop:HO}, \ref{SCQ}, and \ref{Lem:Hull} to this quotient. Parts (a) and (c) of Lemma \ref{Lem:Hull} give us parts (a) and (c) of the theorem, Lemma \ref{SCQ} gives (b), and Lemma \ref{Prop:HO} gives (d).

Since $\mathcal F$ is finite, so is $\mathcal R$. Therefore, $\overline G$ is hyperbolic relative to $\Hl\cup \{ A,B\}$ by part (b) of Lemma \ref{Lem:Hull}. Since $A$ and $B$ are cyclic, $\overline G$ is hyperbolic relative to $\Hl$ by Corollary \ref{Cor:rhh}.
\end{proof}


\section{Automorphisms of property (T) groups acting on hyperbolic spaces}\label{Sec:App}



\subsection{Auxiliary groups}\label{Sec:AuHG}


The main goal of this section is to construct examples of groups with certain special properties, which will be used in the proofs of Theorems \ref{app1} and \ref{app2}. We begin with a variant of the famous Rips construction \cite{Rips}. The modification considered here is similar to those suggested in \cite{BO,OW}.

\begin{lem}\label{Lem:Rips}
Let $Q$ be a countable group, $H$ a non-cyclic, torsion-free, hyperbolic group. There exists a short exact sequence $1\to N\to G\to Q\to 1$ such that the following conditions hold.
\begin{enumerate}
\item[(a)] $G$ is torsion-free.
\item[(b)] $N$ is an infinite quotient of $H$.
\item[(c)] There are elements $a,b\in N$ such that $\langle a\rangle\cap \langle b\rangle=\{ 1\}$ and $\{ \langle a\rangle, \langle b\rangle \} \h G$. In particular, $G$ is acylindrically hyperbolic.
\item[(d)] If, in addition, $Q$ is finitely presented, then $G$ is hyperbolic relative to $\{ \langle a\rangle, \langle b\rangle \}$; in particular, $G$ is a hyperbolic group in this case.
\end{enumerate}
\end{lem}

\begin{proof}
Suppose that $Q$ is given by a (possibly infinite) presentation $\langle Y \mid \mathcal R\rangle$. The free product $P=H\ast F(Y)$, where $F(Y)$ is the free group with basis $Y$, is hyperbolic relative to $\{ H, F(Y)\}$. Since $H$ is non-cyclic, torsion-free, and hyperbolic, it contains two non-commensurable elements $a, b$ such that $E(a)=\langle a\rangle $ and $E(b)=\langle b\rangle $ (see, for example, \cite[Lemma 3.8]{Ols93}). Applying Corollary \ref{Cor:EgRH}, we obtain that $H$ is hyperbolic relative to $\{ \langle a\rangle, \langle b\rangle\}$. By Corollary \ref{Cor:rhrh}, the group $P$ is hyperbolic relative to $\{ \langle a\rangle, \langle b\rangle, F(Y)\}$.

Since $a$ and $b$ are non-commensurable, we have $\langle a\rangle \cap \langle b\rangle =\{ 1\}$. Therefore, $H$ is a suitable subgroup of $P$ with respect to $\{ \langle a\rangle, \langle b\rangle, F(Y)\}$ by Lemma \ref{Lem:suit2} and Lemma \ref{Lem:seq}.

Let $X$ be a finite generating set of $H$ and let $$\mathcal F=\{ y^{-1}xy, yxy^{-1}\mid x\in X, y\in Y\}\cup \mathcal R.$$ Note that $\{ \langle a\rangle, \langle b\rangle\}\h (H,X)$ by Proposition \ref{rhhe}. Therefore, $\{ \langle a\rangle, \langle b\rangle, F(Y)\}\h (P,X)$ by Proposition \ref{trans}. This allows us to apply Theorem \ref{glue} and conclude that, for every $f\in \mathcal F$, there exists $s_f\in H$ such that the quotient group
\begin{equation}\label{Eq:G0}
G=P/\ll fs_f \mid f\in \mathcal F\rr
\end{equation}
is torsion-free, the restriction of the natural homomorphism $P\to G$ to each of the subgroups $\langle a\rangle$, $\langle b\rangle$, $F(Y)$ is injective, and the (isomorphic) images of these subgroups are hyperbolically embedded in $G$. In particular, $G$ is acylindrically hyperbolic by Theorem \ref{Thm:ah}. Let $N$ be the image of $H$ in $G$. The choice of the set $\mathcal F$ ensures that  $N\lhd G$. Since $N$ contains the (isomorphic) image of $\langle a\rangle$, it is infinite.

Finally, if $Q$ is finitely presented, we can assume that $Y$ and $\mathcal R$ are finite. This implies that $|\mathcal F|<\infty$. Thus, part (e) of Theorem \ref{glue} applies and $G$ is hyperbolic relative to the isomorphic images of the subgroups $\langle a\rangle$, $\langle b\rangle$, $F(Y)$. Since each of these subgroups is hyperbolic in this case (recall that $|Y|<\infty$), $G$ is a hyperbolic group by Corollary \ref{Cor:rhh}.
\end{proof}

We are now ready to prove the main result of this section. Note that any pair of groups $N\lhd G$ satisfying conditions (d) and (e) (see below) has the following rigidity property: every automorphism $\phi$ of $N$ preserves the conjugacy classes of $H_1$, $H_2$ in $G$. It follows that $\phi$ maps $x^2$ and $y^3$ to their conjugates (or possibly conjugates of their inverses) in $G$. This property can be used to control the outer automorphism group of appropriate quotients of $N$.

\begin{lem}\label{Thm:HypAut}
Let $Q$ be countable group, $H$ a non-cyclic, torsion-free, hyperbolic group. There exists a short exact sequence $$1\to N\to G\to Q\to 1$$ and subgroups
\begin{equation}\label{Eq:H1H2}
H_1=\langle s,x \mid s^2, \, [s,x]\rangle \cong \ZZ_2\times \ZZ\;\;\; {\rm and}\;\;\; H_2=\langle t,y \mid t^3, \, [t,y]\rangle \cong\ZZ_3\times \ZZ
\end{equation}
of $N$ satisfying the following conditions.
\begin{enumerate}
\item[(a)] We have $H_1\cap H_2=\{ 1\}$ and $\{H_1, H_2\}\h G$. In particular, $G$ is acylindrically hyperbolic.
\item[(b)] $N$ is a quotient of $H$.
\item[(c)] $N$ is generated by $\{ x^2, y^3\}$.
\item[(d)] $H_1=C_G(s)$ and $H_2=C_G(t)$.
\item[(e)] Every finite order element of $N$ is conjugate in $G$ to one of the elements $1, s, t, t^2$.
\item[(f)] If $Q$ is finitely presented, then $G$ is hyperbolic relative to $\{ H_1, H_2\}$; in particular, $G$ is a hyperbolic group in this case.
\end{enumerate}
\end{lem}

\begin{proof}
Let $1\to N_0\to G_0\to Q\to 1$ be the short exact sequence provided by Lemma \ref{Lem:Rips}. That is, $G_0$ is torsion-free and acylindrically hyperbolic, $N_0$ is an infinite quotient of $H$, and $G_0$ is hyperbolic whenever $Q$ is finitely presented.

The group $$P=G_0\ast H_1\ast H_2$$ is hyperbolic relative to $\{ G_0, H_1, H_2\}$. Using the standard facts about free products and the structure of $H_1$, $H_2$, it is easy to show that the following conditions hold:
\begin{enumerate}
\item[(+)]$H_1=C_{P}(s)$, $H_2=C_{P}(t)$;
\item[(++)] \emph{every finite order element in $P$ is conjugate to $1$, $s$, $t$, or $t^2$}.
\end{enumerate}

By Theorem \ref{Thm:ah}, there exists a generating set $\mathcal A_0$ of $G_0$ such that the Cayley graph $\Gamma (G_0, \mathcal A_0) $ is hyperbolic, and the action of $G_0$ on $\Gamma (G_0, \mathcal A_0) $ is non-elementary and acylindrical. Since $G_0$ is torsion free, it contains no non-trivial finite subgroups. Therefore, $N_0$ is $\mathcal A_0$-suitable in $G_0$ by Lemma \ref{Lem:suit1}.

Let $a, b\in N_0$ be the elements provided by Lemma \ref{Lem:Rips}. Thus, we have $\{E(a), E(b)\} \h (G_0, \mathcal A_0)$. Since $\{ G_0, H_1, H_2\}\h P$, we obtain that $\{ E(a), E(b), H_1, H_2\}\h (P, Y)$ for some $Y\subseteq P$ by Theorem \ref{trans}. Let $\mathcal B=Y\sqcup E(a)\sqcup E(b)\sqcup H_1\sqcup H_2$. By Theorem \ref{Thm:AH} (a), we can assume that the action of $P$ on $\Gamma (P, \mathcal B)$ is acylindrical. By Lemma \ref{Lem:suit2}, we conclude that $N_0$ is an $\mathcal B$-suitable subgroup of $P$.

Let $s_1, \ldots, s_k\in N_0$ be the elements provided by Theorem \ref{glue} applied to the group $P$, the hyperbolically embedded collection of subgroups $\{ E(a), E(b), H_1, H_2\}\h (P, Y)$, the $\mathcal B$-suitable subgroup $N_0$, and a finite generating set $\mathcal F=\{ f_1, \ldots, f_k\}$ of $H_1\ast H_2$. Let
\begin{equation}\label{Eq:G1P}
G_1= P/\ll f_1s_1, \ldots, f_ks_k\rr
\end{equation}
We keep the notation $E(a)$, $E(b)$, $H_1$, $H_2$ for the isomorphic images of these subgroups in $G_1$ and denote the image of $N_0$ in $G_1$ by $N_1$. By Theorem \ref{glue}, we have $\{ E(a), E(b), H_1, H_2\}\h (G_1, X_1)$  for some $X_1\subseteq G$. We let
$$
\mathcal A_1=X_1\sqcup E(a)\sqcup E(b)\sqcup H_1\sqcup H_2.
$$
As above, we can assume that the action of $G_1$ on $\Gamma (G_1, \mathcal A_1)$ is acylindrical by part (a) of Theorem~\ref{Thm:AH}. Since $s_1, \ldots, s_k\in N_0$, we have
\begin{equation}\label{H1H2N0}
H_1, H_2\le N_1.
\end{equation}
In particular, the restriction of the natural homomorphism $\gamma_0\colon P\to G_1$ to $G_0$ is surjective. Furthermore, we have $\Ker \gamma_0\le \ll N_0\cup H_1\cup H_2\rr$ (see (\ref{Eq:G1P})) and $\gamma_0(\ll N_0\cup H_1\cup H_2\rr)=N_1$. Therefore,
$$
G_1/N_1 \cong P /\ll N_0 \cup H_1\cup H_2\rr \cong G_0/N_0\cong Q.
$$

By Lemma \ref{Lem:suit2}, the subgroup $S=\langle x^2, y^3\rangle \le G_1$ is $\mathcal A_1$-suitable in $G_1$. Note that $N_1$ is finitely generated being a quotient of $H$. We can apply Theorem \ref{glue} to the group $G_1$ with the collection of hyperbolically embedded subgroups $\{ E(a), E(b), H_1, H_2\}$, the suitable subgroup $S$, and a finite set $\{ g_1, \ldots, g_\ell\} $ of generators of $N_1$. Let $t_1, \ldots, t_\ell\in S$ be the elements provided by the theorem and let
$$
G=G_1/\ll g_1t_1, \ldots, g_\ell t_\ell\rr.
$$
Let $\gamma $ denote the natural homomorphism $G_1\to G$ and let $N=\gamma(N_1)$. Since $g_1, \ldots, g_\ell\in N_1$ and $t_1, \ldots, t_\ell\in S\le N_1$, we have $\Ker(\gamma)\le N_1$. Therefore, $G/N\cong G_1/N_1\cong Q$. Our construction can be summarized by the following commutative diagram, where all rows are exact and all vertical maps are surjective
$$
\begin{tikzcd}
&H\ar[d]&&&\\
1\ar[r]& N_0\ar[r]\ar[d] & G_0 \ar[r]\ar[d,"\gamma_0"] & Q\ar[r]\ar[d, equal]&1 \\
1\ar[r]& N_1\ar[r]\ar[d] & G_1 \ar[r]\ar[d,"\gamma"] & Q\ar[r]\ar[d, equal]&1  \\
1\ar[r]& N\ar[r]& G \ar[r] & Q\ar[r]&1 .
\end{tikzcd}
$$

We claim that the short exact sequence at the bottom satisfies the conclusion of the theorem. Indeed, the groups $H_1$, $H_2$ embed in $G$ since the restriction of $\gamma$ to $H_1\cup H_2$ is injective by Theorem \ref{glue} (a). Moreover, we have $H_1,H_2\le N$ by (\ref{H1H2N0}). Further, parts (a), (d), and (e) follow from the corresponding parts of Theorem \ref{glue} and (+), (++). Part (b) is obvious from the construction. Part (c) is ensured by passing from $G_1$ to $G$.

Finally, part (f) can be proved as follows. If $Q$ is finitely presented, $G_0$ is hyperbolic relative to $\{ E(a), E(b)\}$ by Lemma \ref{Lem:Rips} and $P$ is hyperbolic relative to $\{ E(a), E(b), H_1, H_2\}$ by Corollary \ref{Cor:rhrh}. In particular, $\{ E(a), E(b), H_1, H_2\}\h (P,Y)$ for some finite $Y$ and the action of $P$ on $\Gamma (P, \mathcal B)$ is acylindrical by Proposition \ref{rhhe}.
Thus part (e) of Theorem \ref{glue} applies and we obtain that $G_1$ is hyperbolic relative to (isomorphic image of the) $\{ E(a), E(b), H_1, H_2\}$. Similarly, choosing finite $X_1$ allows us to conclude that $G$ is hyperbolic relative to $\{ E(a), E(b), H_1, H_2\}$. Since each of the subgroups in this collection is hyperbolic, so is $G$ by Corollary \ref{Cor:rhh}.
\end{proof}

\subsection{Proofs of Theorems \ref{app1} and \ref{app2}}

We begin by generalizing some observations used in \cite{BO,OW}.

\begin{defn}\label{Def:IndOut}
To each short exact sequence of groups
\begin{equation}\label{Eq:SES}
1\to N\to G\stackrel{\e}\to Q\to 1,
\end{equation}
we associate homomorphisms
$$
\iota\colon G\to Aut(N)\;\;\; {\rm and}\;\;\; \vk\colon Q\to Out(N)
$$
as follows. For every $g\in G$, $\iota(g)$ is the automorphism of $N$ given by $n\mapsto gng^{-1}$. Further, given an element $q\in Q$, let $g$ be any preimage of $q$ under $\e$. We define
\begin{equation} \label{Eq:vk}
\vk(q)=\iota (g) Inn (N).
\end{equation}
\end{defn}

Note that the map $\vk$ is well-defined. Indeed the right side of (\ref{Eq:vk}) is independent of the choice of a particular preimage $g$.

\begin{lem}\label{Lem:Q->Out}
In the notation of Definition \ref{Def:IndOut}, suppose that $G$ is acylindrically hyperbolic and $K(G)=\{1\}$. If $N$ is non-trivial, then the maps $\iota: G\to Aut(N)$ and $\vk\colon Q\to Out(N)$ associated to the short exact sequence (\ref{Eq:SES}) are injective.
\end{lem}

\begin{proof}
By Theorem \ref{Thm:ah}, there exists a generating set $X$ of $G$ such that $\Gamma (G,X)$ is hyperbolic and the action of $G$ on $\Gamma(G,X)$ is non-elementary and acylindrical. By Lemma \ref{Lem:suit1}, $N$ is an $X$-suitable subgroup of $G$. By Lemma \ref{non-comm} and Theorem \ref{Thm:Eg}, there exist elements $a, b\in N$ of infinite order such that $\{\langle a\rangle , \langle b\rangle\}\h G$.

Suppose that $g\in \Ker (\iota)$. Then $g$ commutes with both $a$ and $b$. By Proposition \ref{Prop:maln}, we have $g\in \langle a\rangle\cap \langle b\rangle=\{ 1\}$. Thus, $\iota$ is injective.

Further, suppose that $q\in \Ker(\vk)$. This means that $\iota (g)$ is an inner automorphism of $N$ for some preimage $g$ of $q$ under $\e$. That is, there is $h\in N$ such that $\iota(g)=\iota(h)$. By injectivity of $\iota$, we have $g=h$. This implies that $q=\e(g)=\e(h)=1$.
\end{proof}

We are now ready to prove our main result. Theorems \ref{app1} and \ref{app2} can be easily derived from this result as explained below.

\begin{thm}\label{main}
For every countable group $Q$ and every non-cyclic torsion-free hyperbolic group $H$, there exists an ICC quotient group $B$ of $H$ with the following properties.
\begin{enumerate}
\item[(a)] $Out(B)\cong Q$.

\item[(b)] The groups $B$ and $Aut(B)$ are acylindrically hyperbolic. Moreover, if $Q$ is finitely presented, then $Aut(B)$ is hyperbolic.
\end{enumerate}
\end{thm}

\begin{proof}
Let $Q$ be a countable group. We fix an exact sequence $1\to N\to G\to Q\to 1$ satisfying conditions (a)--(f) from Lemma \ref{Thm:HypAut}. In particular, $N$ is generated by the elements
\begin{equation}\label{Eq:ab}
a=x^{2}\;\;\; {\rm and } \;\;\; b=y^{3}.
\end{equation}
By Lemma \ref{Thm:HypAut} (a), we have $\{ H_1, H_2\}\h (G,X)$ for some $X\subseteq G$. We keep the standard notation
$$\mathcal A=X\sqcup H_1\sqcup H_2.$$ Let $D$ be the constant provided by Lemma \ref{Omega} applied to the hyperbolically embedded collection $\{H_1, H_2\}\h G$.

Since $C_1(\e, \mu, \rho)$ implies $C(\e, \mu, \rho)$ and both conditions become stronger as $\e$ and $\rho $ increase and $\mu$ decreases, we can choose $\e$ and $\rho$ such that the conclusions of Lemmas \ref{Prop:HO} and \ref{Lem:Hull} simultaneously hold for $N=1$. Further, let $\xi$ and $\sigma $ be the constants provided by Lemma \ref{W->C} for the chosen values of $\e$, $\mu$, and $\rho$. Since $W(\xi, \sigma)$ becomes stronger as $\xi$ and $\sigma$ increase, we can additionally assume (after increasing $\xi$ and $\sigma$ if necessary) that
\begin{equation}\label{Eq:xi>5D}
\xi > 5D.
\end{equation}
and the conclusion of Proposition \ref{Prop:GL} holds for the group $G$, the collection of peripheral subgroups $\{ H_1, H_2\}$, the subset $X\subseteq G$, and $\nu = 1/3$.

Further, by Lemma \ref{Lem:sc}, we can choose $\ell>5$ such that the word
$$
R=a^{\ell+1}b^{\ell+1}a^{\ell+2}b^{\ell+2}\ldots a^{2\ell}b^{2\ell},
$$
where $a$ and $b$ are defined by (\ref{Eq:ab}), satisfies $W(\xi, \sigma)$.

Let
$$
\overline{G}=G/\ll R\rr
$$
and let $B$ be the image of $N$ in $\overline{G}$. Clearly, $B\lhd \overline{G}$ and $\overline{G}/B\cong G/N\cong Q$. By the choice of the parameters, we can apply Lemma \ref{Lem:Hull} to the quotient group $\overline{G}$. Thus, the restriction of the natural homomorphism $G\to \overline G$ to $H_1\cup H_2$ is injective. Henceforth, we identify groups $H_1$, $H_2$ with their (isomorphic) images in $\overline G$ and keep the same notation $s$, $t$, $x$, $y$, $a$, $b$ for their elements. By part (b) of Lemma \ref{Lem:Hull}, we have $\{ H_1, H_2\}\h \overline{G}$. In particular, $\overline G$ is acylindrically hyperbolic.  Since $H_1\cap H_2 =\{ 1\}$ in $G$, we obtain that the images of $H_1$ and $H_2$ intersect trivially in $\overline{G}$. Note that these images actually belong to $B$ since $H_1, H_2\le N$. By Corollary \ref{Cor:ICC}, this implies that $B$ is ICC. Further, by Lemma \ref{Prop:HO} (applied to $N=1$) and part (d) of Lemma \ref{Thm:HypAut}, we have
\begin{equation}\label{CB(s)}
H_1=C_{\overline G}(s)\;\;\; {\rm and}\;\;\; H_2=C_{\overline G}(t)=C_{\overline G}(t^{-1}).
\end{equation}
Finally, every element of finite order in $\overline G$ is conjugate to one of the elements $1$, $s$, $t$, $t^{-1}$ By Lemma \ref{Thm:HypAut} (e) and Lemma \ref{Lem:Hull} (c).

Our next goal is to show that every automorphism of the group $B$ is induced by conjugation in $\overline G$. We fix any
$\phi\in Aut(B)$. Using the description of finite order elements of $\overline G$, we conclude that $\phi$ maps $s$ and $t$ to some conjugates of $s$ and $t^{\pm 1}$ in $\overline G$, respectively. Composing $\phi $ with an inner automorphism of $\overline G$, we can assume that $\phi (s)=s$ and $\phi (t)=zt^{\pm 1}z^{-1}$ for some $z\in \overline G$. Combining this with (\ref{CB(s)}), we obtain $\phi (H_1)=H_1$ and $\phi (H_2)=zH_2z^{-1}$. This easily implies that $x$ is mapped to $x^{\pm 1}$ or $x^{\pm 1}s$ and $y$ is mapped to $y^{\pm 1}$, $y^{\pm 1}t$, or $y^{\pm 1}t^2$ (see (\ref{Eq:H1H2}). In all these cases, we have
$$
\phi (a)=a^{\alpha}\;\;\; {\rm and} \;\;\; \phi (b)=zb^{\beta}z^{-1},
$$
where $\alpha, \beta =\pm 1$.

Note that if $H_1gH_2 =H_1zH_2$ for some $g\in H$, then $z=h_1gh_2$, where $h_1\in H_1$, $h_2\in H_2$. Hence, the pair $(a, zbz^{-1})$ is conjugate to the pair $(a, gbg^{-1})$ by the element $h_1$. Thus, composing $\phi$ with another inner automorphism if necessary, we can assume that $z$ is a shortest element in the double coset $H_1zH_2$. It follows that $z^{-1}$ is a shortest element in the double coset $H_2z^{-1}H_1$.

Let $U$ be a shortest word in the alphabet $\mathcal A$ representing $z$. Consider the word
\begin{equation}\label{Eq:Pi'}
R^\prime= a^{(\ell+1)\alpha}Ub^{(\ell+1)\beta}U^{-1}a^{(\ell+2)\alpha}Ub^{(\ell+2)\beta}U^{-1}\ldots a^{2\ell\alpha}Ub^{2\ell\beta}U^{-1}.
\end{equation}
Here, as usual, we think of powers of $a$ and $b$ as single letters. Clearly $R^\prime$ represents $\phi (R)=1$ in $\overline G$. Thanks to (\ref{Eq:xi>5D}) and the assumption about $z$ explained in the previous paragraph, we can apply Lemma \ref{Wgeod} to the word $R^\prime$ and conclude that $R^\prime $ is $(G,\mathcal A)$-geodesic. By Proposition \ref{Prop:GL}, $R^\prime $ and a cyclic shift of $R^{\pm 1}$ have a common subword $S$ of length at least $(1-\nu) \| R\|=2\| R\|/3$. We will show that this can happen only if $U$ is the empty word.

Note that the letters corresponding to powers of $a$ and $b$ can occur inside the subwords $U^{\pm 1}$ of the word $R^\prime $; the occurrences of these letters that are not inside the subwords $U^{\pm 1}$ will be called \emph{regular}. Suppose first that $S$ contains at most one regular letter. Note that this letter must be the first or the last one in $S$. Indeed, otherwise $S$ would have common subwords with both $U$ and $U^{-1}$ and hence would contain letters corresponding to both positive and negative powers of $a$ or $b$, which is impossible for a subword of $R^{\pm 1}$. Thus $S$ shares a subword $S^\prime$ of length al least $\|S\|-1\ge 2\| R\|/3 -1$ with a single occurrence of the subword $U^{\pm 1}$ in $R^\prime$. Since $\ell >5$, we have $\| R\|>10$ and $2\| R\|/3 -1 > \| R\|/2$. Thus $U=U_1S^\prime U_2$ and there is a cyclic shift of $R$ of the form $S^\prime S^{\prime\prime}$ such that $\| S^{\prime\prime}\| <\| S^\prime\|$.  The word $U_1(S^{\prime\prime})^{-1}U_2$ is shorter than $U$ and represents the same element of the group $\overline G$ as $U$; this contradicts the assumption that $U$ is the shortest word representing the element $z$.

Further, assume that $S$ contains at least $2$ consecutive regular letters. Note that every letter from $H_1\sqcup H_2$ occurs in the words $R$, $R^{-1}$ at most once and no letter occurs in both. This easily implies that $S$ is a subword of $R$ (not $R^{-1}$) and these letters must be next to each other in $R$. Therefore, $\alpha =\beta =1$, $U$ is the empty word, and $z=1$ in this case.

Thus, $\phi (a)=a$ and $\phi (b)=b$ up to conjugation in $\overline G$. Being a quotient of $N$, the group $B$ is generated by $a$ and $b$ (see Lemma \ref{Thm:HypAut} (c)). Therefore, every automorphism of $B$ is induced by a conjugation in $\overline G$. Combining this with Lemma \ref{Lem:Q->Out}, we obtain that $Aut(B)\cong \overline G$ and $Out(B)\cong \overline G/B$. In particular, $Aut(B)$ is acylindrically hyperbolic and $Out(B)\cong Q$.

Finally, if $Q$ is finitely presented, then $G$ is hyperbolic relative to $\{ H_1, H_2\}$ by part (f) of Lemma \ref{Thm:HypAut}. This allows us to apply the second claim in part (b) of Lemma \ref{Lem:Hull} and conclude that $\overline G$ is hyperbolic relative to (the isomorphic images of) $H_1$, $H_2$. Since $H_1, H_2$ are hyperbolic groups, $\overline G$ is hyperbolic by Corollary \ref{Cor:rhh}.
\end{proof}

\begin{proof}[Proof of Theorem \ref{app1}]
Let $H$ be a nontrivial, torsion-free, hyperbolic group with property (T). It is worth noting that such groups are abound; for example, every property (T) group is a quotient of a torsion-free, hyperbolic, property (T) group  by a result of Cornulier \cite{Cor05}. Let also $Q$ be a finite group. By Theorem \ref{main}, there exists an ICC quotient group $G$ of $H$ such that $Aut(G)$ is hyperbolic and $Out(G)\cong Q$. Since $G$ is ICC, we have $G\cong Inn(G)$. Since $Q$ is finite, $Inn(G)$ has finite index in $Aut(G)$. It is well-known that a finite index subgroup of a hyperbolic group is hyperbolic. Thus, $G$ is hyperbolic. It also has property (T) being a quotient of $H$.
\end{proof}

Finally, we state and prove a more precise version of Theorem \ref{app2}.

\begin{thm}\label{app3}
For every countable group $Q$, there exists a finitely generated acylindrically hyperbolic group $G$ with property (T) and trivial finite radical such that $Out(G)\cong Q$.
\end{thm}

\begin{proof}
As above, we apply Theorem \ref{main} to a torsion-free, hyperbolic group $H$ with property (T) and a countable group $Q$. Let $G$ be the group provided by the theorem. It has property (T) being a quotient of $H$. Since $G$ is ICC, it has trivial finite radical and we have $Inn(G)\cong G$. The latter equality implies that $G$ is isomorphic to a normal subgroup of $Aut(G)$. The group $Aut(G)$ is acylindrically hyperbolic by part (b) of Theorem \ref{main}. Therefore, $G$ is also acylindrically hyperbolic by Lemma \ref{Lem:NormAH}.
\end{proof}

\end{document}